\documentclass[11pt,reqno]{amsart}

\usepackage{amsmath,amssymb,mathrsfs,amsthm}
\usepackage{graphicx,cite,cases}
\usepackage{xcolor}
\newtheorem{theorem}{Theorem}[section]

\newtheorem{lemma}[theorem]{Lemma}

\newtheorem{remark}[theorem]{Remark}

\numberwithin{equation}{section}

\begin{document}

\title[open cavity scattering problems]{An adaptive finite element
DtN method for the open cavity scattering problems}

\author{Xiaokai Yuan}
\address{School of Mathematical Science, Zhejiang University,
Hangzhou 310027, China.}
\email{yuan170@zju.edu.cn}

\author{Gang Bao}
\address{School of Mathematical Science, Zhejiang University,
Hangzhou 310027, China.}
\email{baog@zju.edu.cn}

\author{Peijun Li}
\address{Department of Mathematics, Purdue University, West Lafayette, Indiana
47907, USA}
\email{lipeijun@math.purdue.edu}

\thanks{The work of GB is supported in part by an NSFC Innovative Group Fund
(No.11621101). The research of PL is supported in part by the NSF grant
DMS-1912704.}



\begin{abstract}
Consider the scattering of a time-harmonic electromagnetic plane wave by an
open cavity which is embedded in a perfectly electrically conducting infinite
ground plane. This paper concerns the numerical solutions of the open cavity
scattering problems in both transverse magnetic and transverse electric
polarizations. Based on the Dirichlet-to-Neumann (DtN) map for each
polarization, a transparent boundary condition is imposed to reduce the
scattering problem equivalently into a boundary value problem in a bounded
domain. An a posteriori error estimate based adaptive finite element DtN method
is proposed. The estimate consists of the finite element approximation error and
the truncation error of the DtN operator, which is shown to decay exponentially
with respect to the truncation parameter. Numerical experiments are presented
for both polarizations to illustrate the competitive behavior of the adaptive
method.
\end{abstract}

\maketitle

\section{Introduction}

Consider the electromagnetic scattering of a time-harmonic plane wave by an
open cavity, which is referred to as a bounded domain embedded in the ground
with its opening aligned with the ground surface. The open cavity scattering
problems have significant applications in industry and military. In
computational and applied electromagnetics, one of the physical parameter of
interests is the radar cross section (RCS), which measures the detectability of
a target by a radar system. It is crucial to have a deliberate control in the
form of enhancement or reduction of the RCS of a target in the stealth
technology. The cavity RCS caused by jet engine inlet ducts or
cavity-backed patch or slot antennas can dominate the total RCS of an aircraft
or a device. It is indispensable to have a thorough understanding of the
electromagnetic scattering characteristic of a target, particularly a cavity, in
order to successfully implement any desired control of its RCS.

Due to the important applications, the open cavity scattering problems have
received much attention by many researchers in both of the engineering and
mathematics communities. The time-harmonic problems of cavity-backed apertures
with penetrable material filling the cavity interior were introduced and studied
initially by researchers in the engineering community \cite{j-e-1998,
lj-ieee-2000, JV-ieee-91}. The mathematical analysis for the well-posedness of
the variational problems can be found in \cite{abw-2000, abw-2001, abw-2002},
where the non-local transparent boundary conditions, based on the Fourier
transform, were proposed on the open aperture of the cavity. It has been
realized that the phenomena of electromagnetic scattering by cavities not only
have striking physics but also give rise to many interesting mathematical
problems. As more people work on this subject, there has been a rapid
development of the mathematical theory and computational methods for the open
cavity scattering problems. The stability estimates with explicit dependence on
the wavenumber were obtained in \cite{by-arma-2016, byz-sjma-2012}. Various
analytical and numerical methods have been proposed to solve the challenging
large cavity problem \cite{bglz-ieee-2012, bz-ieee-2005, bs-sisc-2005,
wds-nmtma-2008, lms-sinum-2013}. The overfilled cavity problems, where the
filling material inside the cavity may protrude into the space above the ground
surface, were investigated in \cite{hw-ieee-2005, hwh-cicp-2009, LWZ-mmas-2012,
A-JCP-2006}, where the transparent boundary conditions, based on the Fourier
series, were introduced on a semi-circle enclosing the cavity and filling
material. The multiple cavity scattering problem was examined in
\cite{LA-JCP-13,XW-CICP-2016}, where the cavity is assumed to be composed of
finitely many disjoint components. The mathematical analysis can be found in
\cite{bhy-siap-18, dmn-arma-2009} on the related scattering problems in a
locally perturbed half-plane. We refer to the survey \cite{Li-JCM-2018} and the
references cited therein for a comprehensive account on the modeling, analysis,
and computation of the open cavity scattering problems. 

There are two challenges for the open cavity scattering problems: the problems
are formulated in unbounded domains; the solutions may have singularities due to
possible nonsmooth surfaces and discontinuous media. In this paper, we present
an adaptive finite element method with transparent boundary condition to
overcome the difficulties. 

The first issue concerns the domain truncation. The unbounded physical domain
needs to be truncated into a bounded computational domain. An appropriate
boundary condition is required on the artificial boundary of the truncated
domain to avoid unwanted wave reflection. Such a boundary condition is known
as a transparent boundary condition (TBC). There are two different TBCs for the
open cavity scattering problems. For a regular open cavity, where the filling
material is inside the cavity, the Fourier transform based TBC is imposed on the
open aperture of the cavity; for an overfilled cavity, where the filling
material appears to protrude out of the cavity through the open aperture into
the space above the ground surface, the Fourier series based TBC is imposed on
the semi-circle enclosing the cavity and the protruding part. The latter is
adopted in this work since it can be used to handle more general open cavities.
We refer to the perfectly matched layer (PML) techniques \cite{XW-CICP-2016,
zmd-jcm-2009} and the method of boundary integral equations \cite{bgl-nmt-2011}
as alternative approaches for dealing with the issue of the unbounded domains of
the open cavity scattering problems. 

Due to the existence of corners of cavities or the discontinuity of the
dielectric coefficient for the filling material, the solutions have singularites
that slow down the convergence of the finite element for uniform mesh
refinements. The second issue can be resolved by using the a posteriori error
estimate based adaptive finite element method. The a posteriori error estimates
are computable quantities from numerical solutions. They measure the solution
errors of discrete problems without requiring any a priori information of exact
solutions. It is known that the meshes and the associated numerical complexity
are quasi-optimal for appropriately designed adaptive finite element methods. 

The goal of this paper is to combine the adaptive finite element method and the 
transparent boundary conditions to solve the open cavity scattering problems in
an optimal fashion. Specifically, we consider the scattering of a time-harmonic
electromagnetic plane wave by an open cavity embedded in an infinite ground
plane. Throughout, the medium is assumed to be constant in the $x_3$
direction. The ground plane and the cavity wall are assumed to be perfect
electric conductors. The cavity is filled with a nonmagnetic and possibly
inhomogeneous material, which may protrude out of the cavity to the upper
half-space in a finite extend. The infinite upper half-space above the ground
plane and the protruding part of the cavity is composed of a homogeneous medium.
Two fundamental polarizations, transverse magnetic (TM) and transverse electric
(TE), are studied. In this setting, the three-dimensional Maxwell equations may
be reduced to the two dimensional Helmholtz equation and generalized Helmholtz
equation for TM and TE polarizations, respectively. Based on the
Dirichlet-to-Neumann (DtN) map for
each polarization, a transparent boundary condition is imposed to reduce the
scattering problem equivalently into a boundary value problem in a bounded
domain. The nonlocal DtN operator is defined as an infinite Fourier series
which needs to be truncated into a sum of finitely many terms in actual
computation. The a posteriori error estimate is derived bewteen the solution of
original scattering problem and the finite element solution of the discrete
problem with the truncated DtN operator. The error estimate takes account of
the finite element discretization error and the truncation error of the DtN
operator. Using the asymptotic properties of the solution and DtN
operator, we consider a dual problem for the error and show that the truncation
error of the DtN operator decays exponentially respect to the truncation
parameter, which implies that the truncation number does not need to be large. 
Numerical experiments are presented for both polarization cases to demonstrate
the effectiveness of the proposed adaptive method. The related work can be found
in \cite{JLZ-CCP-2013, JLLZ-JSC-2017, JLLWWZ, wbllw-sinum-2015} on the adaptive
finite element DtN method for solving other scattering problems in open
domains. 

The paper is organized as follows. Section \ref{pf} concerns the problem
formulation. The three-dimensional Maxwell equations are introduced and reduced
into the two-dimensional Helmholtz equation under the two fundamental modes:
transverse magnetic (TM) polarization and transverse electric (TE)
polarization. Sections \ref{section:TM} and \ref{section:TE} are devoted to the
TM and TE polarizations, respectively. In each section, the variational problem
and its finite element approximation are introduced; the a posteriori error
analysis is given for the discrete problem with the truncated DtN operator; the
adaptive finite element algorithm is presented. In Section
\ref{section:correction}, the stiff matrix is constructed for the the TBC
part of the sesquilinear form. Section \ref{section:rcs} describes the formulas
of the backscatter radar cross section (RCS). Section \ref{section:numerical}
presents some numerical examples to illustrate the advantages of the proposed 
method. The paper is concluded with some general remarks and directions for
future research in Section \ref{section:conclusion}.

\section{Problem formulation}\label{pf}

Consider the electromagnetic scattering by an open cavity, which is a
bounded domain embedded in the ground with its opening aligned with the ground
surface. By assuming the time dependence $e^{-{\rm i}\omega t}$, the
electromagnetic wave propagation is governed by the time-harmonic Maxwell
equations
\begin{equation}\label{me}
 \nabla\times\boldsymbol E={\rm i}\omega\boldsymbol B,\quad
\nabla\times\boldsymbol H=-{\rm i}\omega\boldsymbol D+\boldsymbol J,
\end{equation}
where $\boldsymbol E$ is the electric field, $\boldsymbol H$ is the magnetic
field, $\boldsymbol B$ is the magnetic flux density, $\boldsymbol D$ is the
electric flux density, $\boldsymbol J$ is the electric current density, and
$\omega>0$ is the angular frequency. For a linear medium, the constitutive
relations, describing the macroscopic properties of the medium, are given by 
\begin{equation}\label{cr}
 \boldsymbol B=\mu\boldsymbol H,\quad \boldsymbol D=\epsilon\boldsymbol E,\quad
\boldsymbol J=\sigma\boldsymbol E,
\end{equation}
where $\mu$ is the magnetic permeability, $\epsilon$ is the electric
permittivity, and $\sigma$ is the electrical conductivity. Throughout, the
medium is assumed to be non-magnetic, i.e., the magnetic permeability $\mu$ is a
constant everywhere, but the electric permittivity $\epsilon$ and the
electrical conductivity $\sigma$ are allowed to be spatial variable
functions. Substituting \eqref{cr} into \eqref{me} leads to a coupled system for
the electric and magnetic fields
\begin{equation}\label{eh}
 \nabla\times\boldsymbol E={\rm i}\omega\mu\boldsymbol H,\quad
\nabla\times\boldsymbol H=-{\rm i}\omega\epsilon\boldsymbol E+\sigma\boldsymbol
E
\end{equation}
Eliminating the magnetic field from \eqref{eh}, we may obtain the Maxwell
system for the electric field
\begin{equation}\label{ef}
 \nabla\times(\nabla\times \boldsymbol E)-\kappa^2\boldsymbol E=0,
\end{equation}
where the wavenumber $\kappa=(\omega^2\epsilon\mu+{\rm
i}\omega\mu\sigma)^{1/2}, \Im\kappa\geq 0$. Similarly, we may eliminate the
electric field and obtain the Maxwell system for the magnetic field 
\begin{equation}\label{mf}
 \nabla\times(\kappa^{-2}\nabla\times\boldsymbol H)-\boldsymbol H=0. 
\end{equation}

When the cavity has a constant cross section along the $x_3$-axis and the
plane of incidence is in the $x_1 x_2$-plane, as a consequence, the
electromagnetic fields are independent of the $x_3$ variable. The
three-dimensional Maxwell equations can be reduced to either the two-dimensional
Helmholtz equation or the two-dimensional generalized Helmholtz equation. 

Let $D\subset\mathbb R^2$ be the cross section of the $x_3$-invariant cavity
with a Lipschitz continuous boundary $\partial D=S\cup\Gamma$. Here $S$ is the
cavity wall and $\Gamma$ is the open aperture of the cavity, which is aligned
with the infinite ground plane $\Gamma_g$. The cavity is filled with an
inhomogeneous medium characterized by the dielectric permittivity $\epsilon$,
the magnetic permeability $\mu$, and the electric conductivity $\sigma$. We
point out that the inhomogeneous medium filling the cavity may protrude into the
space above the ground plane, which is called an overfilled cavity. Let
$B_R^+=\{x\in\mathbb R^2: |x|<R,\,x_2>0\}$ and $B_{\hat R}^+=\{x\in\mathbb
R^2: |x|<\hat R,\,x_2>0\}$ be upper half-discs with radii $R$ and $\hat R$,
where $R>\hat R>0$. Denote by $\Gamma_R^+=\{x\in\mathbb R^2:
|x|=R, x_2>0\}$ and $\Gamma_{\hat R}^+=\{x\in\mathbb R^2:
|x|=\hat R, x_2>0\}$ the upper semi-circles. The radius $\hat R$ can be taken to
be sufficiently large such that the open exterior domain $\mathbb R^2_+\setminus
B_{\hat R}^+$ is filled with a homogeneous medium with constant permittivity
$\epsilon=\epsilon_0$ and zero conductivity $\sigma=0$. Let $\Omega=B_R^+\cup D$
be the bounded domain where our reduced boundary value problems are formulated.
The problem geometry is shown in Figure \ref{pg}. 

\begin{figure}
\centering
\includegraphics[width=0.45\textwidth]{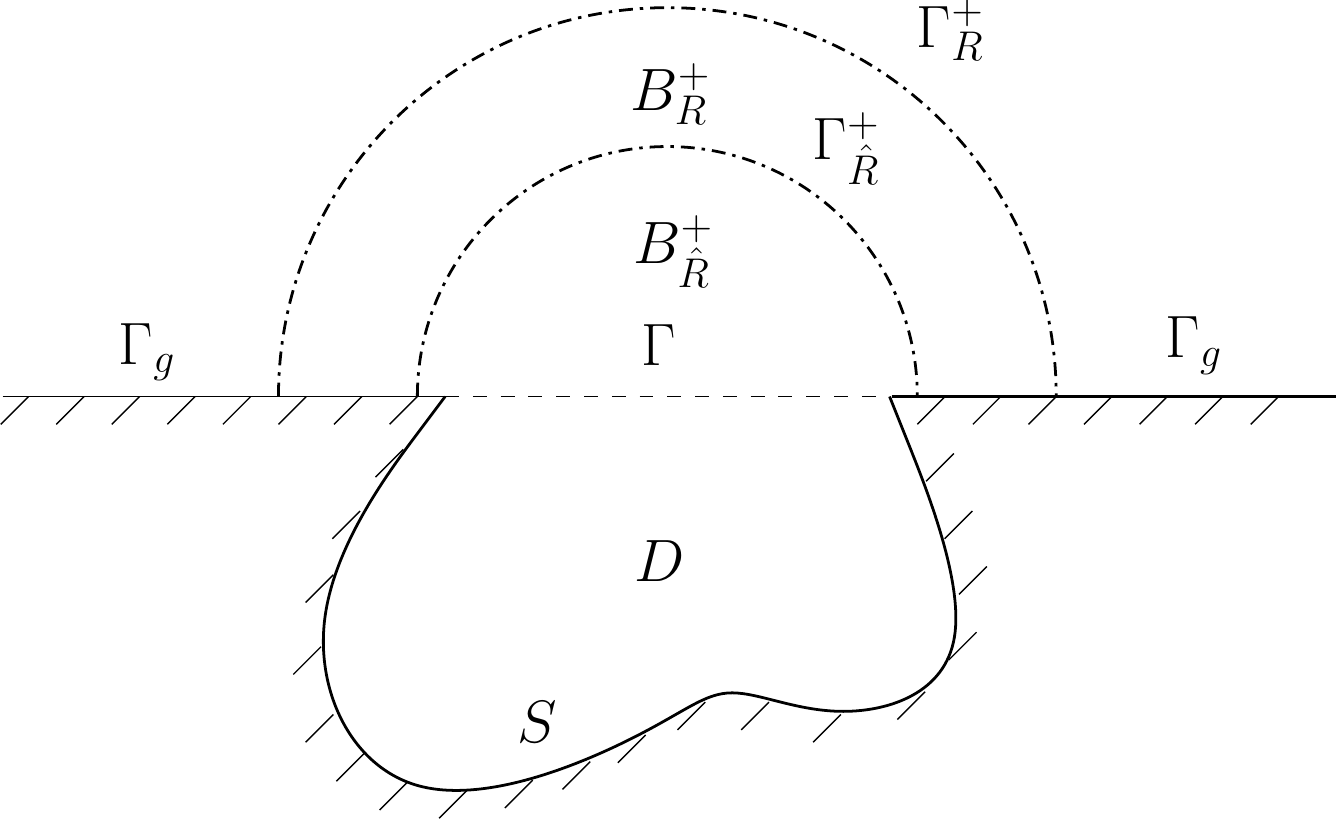}
\caption{Problem geometry of the electromagnetic scattering by an open cavity.}
\label{pg}	
\end{figure}

Since the structure is invariant in the $x_3$-axis, we consider two fundamental
polarizations for the electromagnetic fields: transverse magnetic
(TM) polarization and transverse electric (TE) polarization. In TM case,
the magnetic field is perpendicular to the plane of incidence and does not have
the component in the $x_3$-axis; the electric field, being perpendicular to
the magnetic field and lying in the $x_1 x_2$-plane, is invariant in the
$x_3$-axis and takes the form $\boldsymbol E(x_1, x_2)=(0, 0, u(x_1, x_2))$,
where $u$ is a scalar function. It is easy to verify from \eqref{ef} that $u$
satisfies the Helmholtz equation
\begin{equation}\label{he}
 \Delta u+\kappa^2 u=0\quad\text{in} ~ \mathbb R^2_+\cup D.
\end{equation}
In TE case, the electromagnetic fields are characterized by its electric field
being perpendicular to the plane of incidence and contain no electric field
component in the $x_3$-axis. The magnetic field, being perpendicular to
the electric field and lying in the $x_1 x_2$-plane, is invariant in the
$x_3$-axis and has the form $\boldsymbol H(x_1, x_2)=(0, 0, u(x_1, x_2))$,
where $u$ is also a scalar function. It follows from \eqref{mf} that $u$
satisfies the generalized Helmholtz equation
\begin{equation}\label{ghe}
 \nabla\cdot(\kappa^{-2}\nabla u)+u=0\quad\text{in} ~ \mathbb R^2_+\cup D.
\end{equation}

When the ground plane and the cavity wall are assumed to be perfect conductors,
the following perfectly electrically conducting (PEC) boundary condition can be
imposed
\begin{equation}\label{pec}
 \nu\times\boldsymbol E=0 \quad\text{on}~ \Gamma_g\cup S,
\end{equation}
where $\nu$ is the unit normal vector to $\Gamma_g$ and $S$. In TM polarization,
the PEC boundary condition \eqref{pec} reduces to the homogeneous Dirichlet
boundary condition
\begin{equation}\label{dbc}
 u=0 \quad\text{on}~ \Gamma_g\cup S. 
\end{equation}
In TE polarization, the PEC boundary condition \eqref{pec} reduces to the
homogeneous Neumann boundary condition
\begin{equation}\label{nbc}
 \partial_\nu u=0 \quad\text{on}~ \Gamma_g\cup S. 
\end{equation}

In this paper, we consider the numerical solutions and present an adaptive
finite element DtN method for both of the TM problem \eqref{he}, \eqref{dbc} and
the TE problem \eqref{ghe}, \eqref{nbc}. The more complicated three-dimensional
Maxwell equations will be our future work.

\section{TM polarization}\label{section:TM}

In this section, we discuss the TM polarization and study its finite element
approximation. The a posteriori analysis is carried out for both the finite
element discretization error and the DtN operator truncation error. An adaptive
finite element DtN method is presented for the truncated discrete problem.   

\subsection{Variational problem}

In TM polarization, the nonzero component of the electric field $u$ satisfies
the boundary value problem 
\begin{equation}\label{TM_total}
\begin{cases}
\Delta u+\kappa^2 u=0  \quad &{\rm in} ~ \mathbb R^2_+\cup D,\\
u=0 \quad & {\rm on} ~ \Gamma_g\cup S.
\end{cases}
\end{equation}
Since the problem is imposed in the open domain, a radiation condition is
required to complete the formulation.  

Consider the incidence of a plane wave 
\begin{equation*}\label{incident}
	u^{\rm i}(x_1, x_2)=e^{{\rm i}(\alpha x_1-\beta x_2)},
\end{equation*} 
which is sent from the above to impinge the cavity. Here $\alpha=\kappa_0
\sin\theta, \beta=\kappa_0 \cos\theta$, $\theta\in\left(-\frac{\pi}{2},
\frac{\pi}{2}\right)$ is the incident angle, and
$\kappa_0=\omega(\epsilon_0\mu)^{1/2}$ is the wavenumber in the free space
$\mathbb R^2_+\setminus B_R^+$. It is easy to verify from \eqref{dbc} that the
reflected wave is 
\begin{equation*}\label{TM_reflection}
u^{\rm r}(x_1, x_2)=-e^{{\rm i}(\alpha x_1+\beta x_2)}.
\end{equation*}
By the Jacobi--Anger identity, the incident and reflected waves admit the
following expansions: 
\begin{equation}
u^{\rm i}(x)
=J_0(\kappa_0 r)+2\sum\limits_{n=1}^{\infty} {\rm i}^n J_n(\kappa_0 r)
\cos n(\theta-\pi/2-\phi)\label{Incident_1}
\end{equation}
and
\begin{equation}
u^{\rm r}(x)=-J_0(\kappa_0
r)-2\sum\limits_{n=1}^{\infty} {\rm i}^n
J_n(\kappa_0 r) \cos n(\theta-\pi/2+\phi)\label{TM_reflect1},
\end{equation}
where $J_n$ is the Bessel function of the first kind with order $n$ and
$x=r(\cos\phi,\sin\phi)$ with $\phi$ being the observation angle. Define the
reference wave $u^{\rm ref}=u^{\rm i}+u^{\rm r}$. It follows from
\eqref{Incident_1}--\eqref{TM_reflect1} that 
\begin{equation}
u^{\rm ref}(x) 
= \sum\limits_{n=1}^{\infty} 4{\rm i}^n J_n(\kappa_0 r) \sin n(\theta-\pi/2)
\sin n\phi. \label{TM_ref}
\end{equation}

The total field $u$ consists of the reference field $u^{\rm ref}$ and the
scattered field $u^{\rm s}$, i.e., 
\begin{equation}\label{tf}
u=u^{\rm ref}+u^{\rm s},
\end{equation}
where the scattered field $u^{\rm s}$ is required to satisfy the Sommerfeld
radiation condition
\[
 \lim_{r=|x|\to\infty}r^{1/2}(\partial_{r} u^{\rm s}-{\rm i}\kappa_0
u^{\rm s})=0.
\]

Let $L_{\rm TM}^2(\Gamma_R^{+}):=\left\{u\in L^2(\Gamma_R^{+}): u(R,0)=u(R,
\pi)=0\right\}$. For any $u\in L_{\rm TM}^2(\Gamma_R^+)$, it has the 
Fourier series expansion
\[
u(R,\phi)=\sum\limits_{n=1}^{\infty} a_n \sin n\phi,
\quad a_n=\frac{2}{\pi}\int_{0}^{\pi} u(R,\phi)\sin n\phi\, {\rm d}\phi.
\]
Define the trace function space $H_{\rm TM}^{s}(\Gamma_R^{+}):=\left\{u\in
L_{\rm TM}^2 (\Gamma_R^+): \|u\|_{H_{\rm TM}^s(\Gamma_R^+)}\leq \infty\right\},$
where the $H_{\rm TM}^s(\Gamma_R^+)$ norm is given by 
\[
\|u\|_{H^s_{\rm TM}(\Gamma_R^+)}=\left(\sum\limits_{n=1}^{\infty}
(1+n^2)^s |a_n|^2\right)^{1/2}.
\]
It is clear that the dual space of $H_{\rm TM}^s(\Gamma_R^+)$ is $H_{\rm
TM}^{-s}(\Gamma_R^+)$ with respect to the scalar product in $L^2(\Gamma_R^+)$
given by 
\[
\langle u, v\rangle_{\Gamma_R^+}=\int_{\Gamma_R^+} u\overline{v} \, {\rm d}s.
\]

As discussed in \cite{Li-JCM-2018}, a DtN operator is introduced on
$\Gamma_R^+$: 
\begin{equation}\label{TM_DtN}
(B_{\rm TM}u)(R,\phi)=\kappa_0\sum\limits_{n=1}^{\infty}
\frac{H_n^{(1)'}(\kappa_0 R)}{H_n^{(1)}(\kappa_0 R )}a_n\sin n\phi,
\end{equation}
where $H_n^{(1)}$ is the Hankel function of the first kind with order $n$. It is
shown in \cite[Lemma 3.1]{A-JCP-2006} that $B_{\rm TM}: H^{1/2}_{\rm
TM}(\Gamma_R^+)\rightarrow H_{TM}^{-1/2}(\Gamma_R^+)$ is continuous. The TBC can
be imposed for the total field as follows:
\[
\partial_{\rho}u=B_{\rm TM} u+f\quad {\rm on} ~ \Gamma_R^+,
\]
where $f=\partial_{\rho} u^{\rm ref}-B_{\rm TM}u^{\rm ref}$. Substituting
\eqref{TM_ref} into \eqref{TM_DtN} and applying the Wronskian identity, we
obtain explicitly
\begin{eqnarray*}
 f = -\frac{8}{\pi R} \sum\limits_{n=1}^{\infty}\frac{{\rm
i}^{n+1}}{H_n^{(1)}(\kappa_0 R)}\sin n(\theta-\pi/2)\sin n\phi.
\end{eqnarray*}

The original cavity scattering problem \eqref{TM_total} can be reduced
equivalently into the boundary value problem 
\begin{equation*}
\begin{cases}
\Delta u+k^2 u=0\quad & {\rm in} ~ \Omega,\\
u=0\quad & {\rm on} ~ S\cup \Gamma_g,\\
\partial_{\rho} u-B_{\rm TM} u=f \quad & {\rm
on} ~ \Gamma_R^+,
\end{cases}
\end{equation*}
which has the variational formulation: find $u\in H_0^1(\Omega):=\left\{u\in
H^1(\Omega): u=0 ~ {\rm on} ~ S\cup \Gamma_g\right\}$ such that
\begin{equation}\label{TM_v1}
a_{\rm TM} (u, v)=\int_{\Gamma_R^+} f\bar{v}\,{\rm d}s\qquad \forall\,v\in
H_0^1(\Omega).
\end{equation}
Here the sesquilinear form $a_{\rm TM}: H_0^1(\Omega)\times
H_0^1(\Omega)\to\mathbb C$ is defined as
\[
a_{\rm TM} (u, v)=\int_{\Omega} \nabla u\cdot \nabla\bar{v}\,{\rm
d}x-\int_{\Omega}\kappa^2 u\bar{v}\,{\rm
d}x-\int_{\Gamma_R^+}B_{\rm TM} u\,\bar{v}\,{\rm d}s.
\]

The following theorem of the well-posedness of the variational problem
\eqref{TM_v1} is proved in \cite{Li-JCM-2018}.

\begin{theorem}
The variational problem \eqref{TM_v1} has a unique solution $u\in
H_0^1(\Omega)$, which satisfies the estimate
\[
\|u\|_{H^1(\Omega)}\lesssim \|f\|_{H^{-1/2}(\Gamma^{+}_R)}.
\]
\end{theorem}

Hereafter, the notation $a\lesssim b$ stands for $a\leq Cb$, where $C$
is a positive constant whose value is not required but should be clear from the
context. 

\subsection{Finite element approximation}

Let $\mathcal M_h$ be a regular triangulation of $\Omega$, where $h$ denotes the
maximum diameter of all the elements in $\mathcal M_h$. To avoid 
being distracted from the main focus of the a posteriori error analysis, we
assume for simplicity that $S$ and $\Gamma_R^+$ are polygonal to keep from using
the isoparametric finite element space and deriving the approximation error of
the boundaries $S$ and $\Gamma_R^+$. Thus any edge $e\in\mathcal M_h$ is a
subset of $\partial \Omega$ if it has two boundary vertices.

Let $V_h\subset H_0^1(\Omega)$ be a conforming finite element space, i.e., 
\[
V_h:=\left\{v_h\in C(\overline{\Omega}): v_h|_T\in P_m(K) ~
\forall\, T\in \mathcal M_h, v_h=0 ~ {\rm on} ~S\cup\Gamma_g\right\}.
\]

In practice, the DtN operator \eqref{TM_DtN} needs to be truncated into a sum
of finitely many terms
\begin{equation}\label{TM_tDtN}
B_{\rm TM}^{N} u=\kappa_0\sum\limits_{n=1}^{N} 
\frac{H_n^{(1)'}(\kappa_0 R)}{H_n^{(1)}(\kappa_0 R)} a_n \sin
n\phi,\quad a_n=\frac{2}{\pi}\int_{0}^{\pi} u(R, \phi)\sin n\phi\,{\rm
d}\phi.
\end{equation}
Taking account of the DtN operator truncation, we obtain the finite element
approximation to the variational problem \eqref{TM_v1}: find $u_h\in V_h$
such that 
\begin{equation}\label{TM_v2}
a_{\rm TM}^N (u^h, v^h)=\int_{\Gamma_R^+} f\,\overline{v^h}\, {\rm
d}s\quad\forall~v_h\in V_h,
\end{equation}
where the sesquilinear form $a_{\rm TM}^N: V_h\times V_h\to\mathbb C$ is
\[
a_{\rm TM}^N (u^h, v^h)=
\int_{\Omega} \nabla u^h\cdot \nabla \overline{v^h}\,{\rm
d}x-\int_{\Omega}\kappa^2 u^h\,\overline{v^h}\,{\rm d}x
-\int_{\Gamma_R^+}B_{\rm TM}^{N} u^h \,\overline{v^h}\,{\rm d}s.
\]

For sufficiently small $h$ and sufficiently large $N$, the discrete inf-sup
condition of the sesquilinear form $a_{\rm TM}^{N}$ can be established by an
argument of Schatz \cite{S-mc74}. It follows from the general theory in
\cite{BA-AP-1973} that the truncated variational problem \eqref{TM_v2} admits a
unique solution. Since our focus is the a posteriori error estimate and
the associated adaptive algorithm, we assume that the discrete problem
\eqref{TM_v2} has a unique solution $u_h\in V_h$.

\subsection{A posteriori error analysis}

For any triangular element $T\in\mathcal M_h$, denoted by $h_T$
its diameter. Let $\mathcal B_h$ denote the set of all the edges of $T$. For
any edge $e\in\mathcal B_h$, denote by $h_e$ its length. For any interior edge
$e$, which is the common side of triangular elements $T_1, T_2\in\mathcal M_h$,
we define the jump residual across $e$ as
\[
J_e =-\big(\nabla u^h|_{T_1}\cdot \nu_1+\nabla u^h|_{T_2}\cdot
\nu_2\big), 
\]
where $\nu_{j}$ is the unit outward normal vector on the boundary of $T_j, j=1,
2$. For any boundary edge $e\subset \Gamma_R^+$, the jump residual is defined as
\[
J_e=2\big(B_{\rm TM}^N u^h-\nabla u^h\cdot \nu-f\big).
\]
For any triangle $T\in\mathcal M_h$, denote by $\eta_T$ the local error
estimator as follows: 
\[
\eta_T=h_T \|H_{\rm TM}
u^h\|_{L^2(T)}+\bigg(\frac{1}{2}\sum\limits_{e\in\partial T}h_e
\|J_e\|_{L^2(e)}^2\bigg)^{1/2}
\]
where $H_{\rm TM}$ is the Helmholtz operator defined by $H_{\rm TM} u=\Delta
u+\kappa^2 u$.

Let $\xi=u-u^h$, where $u$ and $u_h$ are the solutions of the variational
problems \eqref{TM_v1} and \eqref{TM_v2}, respectively. Introduce a dual
problem: find $w\in H_0^1(\Omega)$ such that 
\begin{equation}\label{TM_dual}
a_{\rm TM}(v, w)=\int_{\Omega} v \bar{\xi}\,{\rm d}x\quad \forall v\in
H_0^1(\Omega).
\end{equation}
It is easy to check that $w$ is the solution of the following boundary value
problem: 
\begin{equation*}
\begin{cases}
\Delta w+\kappa^2 w=-\xi\quad & {\rm in} ~ \Omega,\\
w=0 \quad & {\rm on} ~ S\cup \Gamma_g,\\
\partial_{\rho} w=B_{\rm TM}^* w \quad & {\rm on} ~  
\Gamma_R^+,
\end{cases}
\end{equation*}
where $B_{\rm TM}^*$ is the adjoint operator of $B_{\rm TM}$ and is given by 
\[
B_{\rm TM}^* u=\kappa_0\sum\limits_{n=1}^{\infty} 
\overline{\left(\frac{H_n^{(1)'}(\kappa_0 R)}{H_n^{(1)}(\kappa_0 R)}\right)} a_n
\sin n\phi,\quad a_n=\frac{2}{\pi}\int_{0}^{\pi} u(R, \phi)\sin
n\phi\,{\rm d}\phi. 
\]

The following three lemmas are proved in \cite{JLLZ-JSC-2017}, where the first
lemma concerns the well-posedness of the dual problem, the second lemma gives
the trace result in $H_0^1(\Omega)$, and the third lemma shows the error
representation formulas. 

\begin{lemma}\label{Thm_TM_dual}
The dual problem \eqref{TM_dual} has a unique solution $w\in H^1_0(\Omega)$,
which satisfies the estimate
\[
\|w\|_{H^1(\Omega)}\lesssim \|\xi\|_{L^2(\Omega)}.
\] 
\end{lemma}

\begin{lemma}\label{Lema_TM_a}
For any $u\in H_0^{1}(\Omega)$, the following estimates hold: 
\[
\|u\|_{H^{1/2}(\Gamma_R^+)}\lesssim \|u\|_{H^{1}(\Omega)}, \quad
\|u\|_{H^{1/2}(\Gamma_{\hat R}^+)}\lesssim \|u\|_{H^{1}(\Omega)}.
\]
\end{lemma}

\begin{lemma}\label{Lema_TM_b}
Let $u, u_h$, and $w$ be the solutions to the problems \eqref{TM_v1}, 
\eqref{TM_v2}, and \eqref{TM_dual}, respectively. The following identities hold:
\begin{eqnarray*}
&&\|\xi\|^2_{H^1(\Omega)}=\Re\Big(a_{\rm TM}(\xi, \xi)+\langle(B_{\rm
TM}^N-B_{\rm TM})\xi, \xi\rangle_{ \Gamma_R^+}\Big) + \Re\langle B^N_{\rm TM}
\xi, \xi\rangle_{\Gamma_R^+}+\Re\int_{\Omega}(\kappa^2+1)|\xi|^2{\rm d}x,\\
&&\|\xi\|^2_{L^2(\Omega)}=a_{\rm TM}(\xi, w)+\langle(B_{\rm
TM}-B_{\rm TM}^N)\xi, w\rangle_{\Gamma_R^+}
-\langle(B_{\rm TM}-B_{\rm TM}^N)\xi, w\rangle_{\Gamma_R^+},\\
&&a_{\rm TM}(\xi, \psi)+\langle(B_{\rm TM}-B^N_{\rm
TM})\xi, \psi\rangle_{\Gamma_R^+}=\int_{\Gamma_R^+}
f (\overline{\psi-\psi_h})\,{\rm d}s-a_{\rm TM}^N(u^h, \psi-\psi^h)\notag\\
&&\qquad\qquad\quad  +\langle (B_{\rm TM} -B^N_{\rm TM})u,
\psi\rangle_{\Gamma_R^+}\quad\forall \,\psi\in H^1_0(\Omega),  \psi_h\in V_h. 
\end{eqnarray*}
\end{lemma}

The following result concerns the truncation error of the DtN operator and plays
an important role in the a posteriori error estimate. 

\begin{lemma}\label{TM_lemmatotal}
Let $u$ be the solution to \eqref{TM_v1} and $\psi$ be any function in
$H^1_0(\Omega)$. For sufficiently large $N$, the following estimate holds: 
\[
\left|\langle(B_{\rm TM}-B_{\rm TM}^N)u,
\psi\rangle_{\Gamma_R^+}\right|\lesssim 
\bigg[\Big(\frac{\hat{R}}{R}\Big)^N+\Big(\frac{e\kappa_0
R}{2N}\Big)^{2N+4}\bigg]
\|u^{\rm ref}\|_{H^1(\Omega)}\|\psi\|_{H^1(\Omega)}.
\]
\end{lemma}

\begin{proof}
By \eqref{tf}, we have $u=u^{\rm s}+u^{\rm ref}$, where $u^{\rm ref}$ is the
reference field and $u^{\rm s}$ is the scattered field satisfying the Sommerfeld
radiation condition. For sufficiently large $N$, it is shown in \cite[Lemma
4]{JLLZ-JSC-2017} that
\begin{eqnarray*}
|\langle(B_{\rm TM}-B_{\rm TM}^N)u^{\rm s}, \psi\rangle_{\Gamma_R^+}|\lesssim 
\Big(\frac{\hat R}{R}\Big)^N \|u^{\rm
s}\|_{H^{1/2}(\Gamma_R^+)}\|\psi\|_{H^{1/2}(\Gamma_R^+)}.
\end{eqnarray*}

A straightforward calculation yields 
\begin{eqnarray*}%\label{TM_ref_1}
(B_{\rm TM}-B_{\rm TM}^{N})u^{\rm ref}=
\kappa_0 \sum\limits_{n=N+1}^{\infty}\frac{H_n^{(1)'}(\kappa_0 R)}{
H_n^{(1)}(\kappa_0 R)}\left[4 {\rm i}^n J_n(\kappa_0 R) \sin
n(\theta-\frac{\pi}{2})\right] \sin n\phi.
\end{eqnarray*}
By \cite{w-22}, for sufficiently large $n$, we have 
\[
J_n(z)\sim \frac{1}{\sqrt{2\pi n}}\left(\frac{ez}{2n}\right)^n, \quad 
\left|\frac{H_n^{(1)'}(\kappa_0 R)}{H_n^{(1)}(\kappa_0 R)}\right|\lesssim n,
\]
which give
\begin{eqnarray*}
	&& |\langle(B_{\rm TM}-B_{\rm TM}^{N})u^{\rm ref}, \psi\rangle|
		=\frac{\pi}{2}\kappa_0\left|\sum\limits_{n=N+1}^{\infty}\frac{H_n^{(1)'}(\kappa_0 R)}{
		H_n^{(1)}(\kappa_0 R)}\left[4{\rm i}^n J_n(\kappa_0 R) \sin
n(\theta-\frac{\pi}{2})\right] \hat{\psi}_n(R)\right|\\
	&& \leq 2\pi\kappa_0 \sum\limits_{n=N+1}^{\infty}
		\left|\frac{H_n^{(1)'}(\kappa_0 R)}{H_n^{(1)}(\kappa_0 R)}\right|
		\left| J_n(\kappa_0 R)\right| | \hat{\psi}_n(R)|\\
	&& \lesssim 2\pi\kappa_0\sum\limits_{n=N+1}^{\infty}
		n\frac{1}{\sqrt{2\pi n}}\left(\frac{e\kappa_0 R}{2n}\right)^n
		|\hat{\psi}_n(R)|\\
	&& \lesssim \sqrt{2\pi}\kappa_0\left\{\sum\limits_{n=N+1}^{\infty}
		\left[\sqrt{n}\frac{1}{\sqrt{2\pi n}}
		\left(\frac{e\kappa_0 R}{2n}\right)^n\right]^2\right\}^{1/2}
		\left\{\sum\limits_{n=N+1}^{\infty} 2\pi \left(1+n^2\right)^{1/2}
		|\hat{\psi}_n(R)|^2\right\}^{1/2}\\
	&& \lesssim \kappa_0 \left\{\sum\limits_{n=N+1}^{\infty}
		\left(\frac{e\kappa_0 R}{2n}\right)^{2n}\right\}^{1/2}
		\|\psi\|_{H^{1/2}(\Gamma_R^+)}.
\end{eqnarray*}
For $N>\frac{e\kappa_0 R}{2}$, it is easy to verify
\[
\sum\limits_{n=N+1}^{\infty}
\left(\frac{e\kappa_0 R}{2n}\right)^{2n}\leq\sum\limits_{n=N+1}^{\infty}
\left(\frac{e\kappa_0 R}{2N}\right)^{2n}=
\frac{\left(\frac{e\kappa_0 R}{2N}\right)^{2N+4}}
{1-\left(\frac{e\kappa_0 R}{2N}\right)^{2}}.
\]
Hence
\begin{eqnarray*}
 |\langle(B_{\rm TM}-B_{\rm TM}^{N})u^{\rm ref},
\psi\rangle|
&\lesssim& \kappa_0 \frac{\left(\frac{e\kappa_0 R}{2N}\right)^{2N+4}}
	{1-\left(\frac{e\kappa_0 R}{2N}\right)^{2}}\sqrt{2\pi R^2}
	\sqrt{\frac{1}{2\pi R^2}}\|\psi\|_{H^{1/2}(\Gamma_R^+)}\\
& \lesssim& \frac{\kappa_0}{\sqrt{2\pi} R}	
	\frac{\left(\frac{e\kappa_0 R}{2N}\right)^{2N+4}}
	{1-\left(\frac{e\kappa_0 R}{2N}\right)^{2}}\|u^{\rm ref}\|_{H^1(B_R^+)}
	\|\psi\|_{H^1(\Omega)}\\
&\leq& \frac{\kappa_0}{\sqrt{2\pi} R}	
	\frac{1}{1-\left(\frac{e\kappa_0 R}{2N}\right)^{2}}
	\left(\frac{e\kappa_0 R}{2N}\right)^{2N+4}\|u^{\rm ref}\|_{H^1(\Omega)}
		\|\psi\|_{H^1(\Omega)}.
\end{eqnarray*}
Combining the above estimates, we obtain 
\[
|\langle(B_{\rm TM}-B_{\rm TM}^{N})u,
\psi\rangle|\lesssim \left(\frac{R'}{R}\right)^N 
\|u^{\rm s}\|_{H^{1/2}(\Gamma_R^+)}\|\psi\|_{H^{1}(\Omega)}
+\left(\frac{e\kappa_0 R}{2N}\right)^{2N+4}\|u^{\rm ref}\|_{H^1(\Omega)}
\|\psi\|_{H^1(\Omega)}.	
\]

Since 
\[
\|u^s\|_{H^1(\Omega)}=\|u-u^{\rm ref}\|_{H^1(\Omega)}\leq
\|u\|_{H^1(\Omega)}+\|u^{\rm ref}\|_{H^1(\Omega)}
\]
and 
\[
\|u\|_{H^1(\Omega)}\lesssim \|f\|_{H^{-1/2}(\Gamma_R^+)},
\]
it suffices to estimate $f$. A simple calculation yields 
\begin{eqnarray*}
\|f\|_{H^{-1/2}(\Gamma_R^+)} &\leq& 
\|\partial_{\rho} u^{\rm ref}\|_{H^{-1/2}(\Gamma_R^+)}
+\|B_{\rm TM} u^{\rm ref}\|_{H^{-1/2}(\Gamma_R^+)}\\
&\lesssim& \|\partial_{\rho} u^{\rm ref}\|_{H^{-1/2}(\Gamma_R^+)}
+\|u^{\rm ref}\|_{H^{1/2}(\Gamma_R^+)}.
\end{eqnarray*}
Taking the normal derivative of \eqref{TM_ref}, we get
\[
\partial_{\rho} u^{\rm ref}=4\sum\limits_{n=1}^{\infty}{\rm i}^n \kappa_0
J_n'(\kappa_0 R)\sin n(\theta-\frac{\pi}{2})\sin n\phi.
\]
It follows from the definition of the norm on $H^{-1/2}(\Gamma_R^+)$ that 
\begin{eqnarray*}
&&\|\partial_{\rho} u^{\rm ref}\|_{H^{-1/2}(\Gamma_R^+)}
\lesssim 2\pi \left\{\sum\limits_{n=1}^{\infty}\left(1+n^2\right)^{-1/2}
\kappa_0^2 \left|J_n'(\kappa_0 R)\sin n(\theta-\frac{\pi}{2})\right|^2
\right\}^{1/2}\\	
&&\lesssim 2\pi \left\{\sum\limits_{n=1}^{\infty}\left(1+n^2\right)^{-1/2}
\kappa_0^2 \left|\frac{J_n'(\kappa_0 R)}{J_n(\kappa_0 R)}
\right|^2 \left|J_n(\kappa_0 R)\right|^2\left|\sin
n(\theta-\frac{\pi}{2})\right|^2\right\}^{1/2}.
\end{eqnarray*}
For sufficiently large $n$, we have from the asymptotic
property of the Bessel function $J_n$ (cf. \cite{w-22}) that 
\[
\frac{J_n'(z)}{J_n(z)}=\frac{J_{n-1}(z)}{J_n(z)}-\frac{n}{z}
\sim \frac{\sqrt{n}}{\sqrt{n-1}}\left(\frac{n}{n-1}\right)^{n-1}
\frac{2n}{ez}-\frac{n}{z}\sim \frac{n}{z}.
\]
Hence
\begin{eqnarray*}
\|\partial_{\rho} u^{\rm ref}\|_{H^{-1/2}(\Gamma_R^+)}
&\lesssim& 2\pi \left\{\sum\limits_{n=1}^{\infty}\left(1+n^2\right)^{-1/2}
\kappa_0^2 n^2 \left|J_n(\kappa_0 R)\sin
n(\theta-\frac{\pi}{2})\right|^2\right\}^{1/2}\\	
&\lesssim& \left\{2\pi \sum\limits_{n=1}^{\infty}
\left(1+n^2\right)^{1/2}\kappa_0^2\left|J_n(\kappa_0 R)
\right|^2\left|\sin n(\theta-\frac{\pi}{2})\right|^2 \right\}^{1/2}\\
&=& \|u^{\rm ref}\|_{H^{1/2}(\Gamma_R^+)}.
\end{eqnarray*}
Noting
\[
\|f\|_{H^{-1/2}(\Gamma_R^+)}\lesssim \|u^{\rm
ref}\|_{H^{1/2}(\Gamma_R^+)}\lesssim 
\|u^{\rm ref}\|_{H^1(\Omega)},
\]
we complete the proof. 
\end{proof}

\begin{remark}
We notice that the result and proof of Lemma \ref{TM_lemmatotal} is 
different from those for the scattering problems in periodic structures
\cite{JLLWWZ, wbllw-sinum-2015, LY-CMAME-2020}. For the latter problems, the
DtN operators are defined on a straight line or plane surface and have only
finitely many terms when acting on the incident fields. For our case, the DtN
operator is defined on a semi-circle and is still an infinite series when
acting on the reference field, which results in an extra term in the estimate
given in Lemma \ref{TM_lemmatotal}. 
\end{remark}

\begin{lemma}
Let $w$ be the solution of the dual problem \eqref{TM_dual}. Then the
following estimate holds:
\[
|\langle(B_{\rm TM}-B_{\rm TM}^N)\xi,
w\rangle_{\Gamma_R^+}|\leq N^{-2} \|\xi\|_{H^1(\Omega)}^2.
\]
\end{lemma}

\begin{proof}
Since $w=0$ on $\Gamma_g$, it admits the Fourier series expansion in
terms of the sin functions
\[
w(r, \phi)=\sum\limits_{n=1}^{\infty} \hat{w}^{(n)}(r)\sin n\phi,\quad r\in
[R', R],  \phi\in[0, \pi],
\] 
where $w^{(n)}(r)$ are the Fourier coefficients. Following the same proof
as that in \cite[Lemma 5]{JLLZ-JSC-2017}, we may show the desired result. 
\end{proof}

The following theorem presents the a posteriori error estimate and is the main
result for the TM polarization. By Lemma \eqref{TM_lemmatotal}, the proof is
essentially the same as that for \cite[Theorem 1]{JLLZ-JSC-2017}. The details
are omitted for brevity. 

\begin{theorem}\label{TM_mainthm}
Let $u$ and $u_h$ be the solution of \eqref{TM_v1} and \eqref{TM_v2}, respectively. There exists a positive integer $N_0$ 
independent of $h$ such that for $N>N_0$, the following a posteriori error
estimate holds: 
\[
\|u-u^h\|_{H^1(\Omega)}\lesssim \bigg(\sum\limits_{T\in
M_h}\eta_T^2\bigg)^{1/2} +\bigg[\Big(\frac{\hat{R}}{R}\Big)^N 
+\Big(\frac{e\kappa_0 R}{2N}\Big)^{2N+4}\bigg]\|u^{\rm
ref}\|_{H^1(\Omega)}.
\]
\end{theorem}

\subsection{Adaptive FEM algorithm}

It is shown in Theorem \ref{TM_mainthm} that the a posteriori error consists of two parts: the finite element discretization error 
$\varepsilon_h$ and the DtN operator truncation error $\varepsilon_N$, where 
\begin{eqnarray}\label{TM_epsilonN}
\varepsilon_h = \bigg(\sum\limits_{K\in\mathcal M_h}
\eta^2_{T}\bigg)^{1/2},\quad 
\varepsilon_N =\bigg[\Big(\frac{\hat{R}}{R}\Big)^N 
+\Big(\frac{e\kappa_0 R}{2N}\Big)^{2N+4}\bigg]\|u^{\rm
ref}\|_{H^1(\Omega)}. 
\end{eqnarray}
In the implementation, based on \eqref{TM_epsilonN}, the parameters $\hat R, R$,
and $N$ can be chosen appropriately such that the finite element discretization
error is not contaminated by the truncation error, i.e., $\varepsilon_N$ is
required to be small compared with $\varepsilon_h$, for instance,
$\varepsilon_N\leq 10^{-8}$. Table \ref{algo} shows the algorithm of the
adaptive finite element DtN method for solving the open cavity scattering
problem in the TM polarization.  

\begin{table}
\caption{The adaptive finite element DtN method for TM polarization.}
%\vspace{1ex}
\hrule \hrule
\vspace{0.8ex} 
\begin{enumerate}		
\item Given the tolerance $\varepsilon>0$ and the parameter $\tau\in(0,1)$.
\item Fix the computational domain $\Omega$ by choosing $R$.
\item Choose $R'$ and $N$ such that $\epsilon_N\leq 10^{-8}$.
\item Construct an initial triangulation $\mathcal M_h$ over $\Omega$ and compute
error estimators.
\item While $\varepsilon_h>\varepsilon$ do
\item \quad refine mesh $\mathcal M_h$ according to the strategy \\
\[\text{if } \eta_{\hat{T}}>\tau \max\limits_{T\in\mathcal M_h}
\eta_{T}, \text{ refine the element } \hat{T}\in\mathcal M_h, \]
\item \quad denote refined mesh still by $\mathcal M_h$, solve the discrete
problem \eqref{TM_v2} on the new mesh $\mathcal M_h$,
\item \quad compute the corresponding error estimators.
\item End while.
\end{enumerate}
\vspace{0.8ex}
\hrule\hrule
\vspace{0.8ex}
\label{algo}
\end{table}

\section{TE polarization}\label{section:TE} 

In this section, we consider the TE polarization. Since the discussions are
similar to the TM polarization, we briefly present the parallel results without
providing the details. 

In TE polarization, the total field $u$ satisfies the boundary value problem of
the generalized Helmholtz equation
\begin{equation}\label{TE_total}
\begin{cases}
\nabla\cdot(\kappa^{-2}\nabla u)+u=0  \quad
&{\rm in} ~ \Omega,\\
\partial_{\nu} u=0 \quad & {\rm on} ~ S\cup \Gamma_g.
\end{cases}
\end{equation}

Consider the same plane incident wave $u^{\rm i}(x_1,x_2)=e^{{\rm i}(\alpha
x_1-\beta x_2)}$. Due to the homogeneous Neumann boundary condition on
$\Gamma_g$, the reflected field is
\begin{equation*}\label{TE_reflection}
u^{\rm r}(x_1, x_2)=e^{{\rm i}(\alpha x_1+\beta x_2)}.
\end{equation*}
By the Jacobi--Anger identity, the reference wave $u^{\rm ref}=u^{\rm
i}+u^{\rm r}$ admits the following expansion:
\begin{eqnarray}
u^{\rm ref}(x_1, x_2) =2J_0(\kappa_0 r)+4\sum\limits_{n=1}^{\infty} {\rm i}^n
J_n(\kappa_0 r) \cos n(\theta-\pi/2) \cos n\phi.\label{TE_ref}
\end{eqnarray}

Once again, the total field $u$ is assumed to be composed of the reference field
$u^{\rm ref}$ and the scattered field $u^{\rm s}$, i.e., 
\[
u=u^{\rm ref}+u^{\rm s},
\]
where the scattered field $u^{\rm s}$ is also required to satisfy the Sommerfeld
radiation condition
\[
 \lim_{r=|x|\to\infty}r^{1/2}(\partial_{r} u^{\rm s}-{\rm i}\kappa_0
u^{\rm s})=0.
\]

Let $u\in L_{\rm TE}^2(\Gamma_R^{+})=\left\{u\in L^2(\Gamma_R^{+}): 
\partial_{\phi}u(R,0)=\partial_{\phi}u(R, \pi)=0\right\}$. For any $u\in
L_{\rm TE}^2$, it has the Fourier series expansion
\[
u(R,\phi)=\sum\limits_{n=0}^{\infty} a_n \cos n\phi,
\]
where 
\[
a_0=\frac{1}{\pi}\int_{0}^{\pi} u(R,\phi)\,{\rm d}\phi,
\quad a_n=\frac{2}{\pi}\int_{0}^{\pi} u(R,\phi)\cos n\phi\,{\rm d}\phi.
\]
Define the trace function space $H_{\rm TE}^{s}(\Gamma_R^{+})=\left\{u\in
L_{\rm TE}^2 (\Gamma_R^+): 
\|u\|_{H_{\rm TE}^s(\Gamma_R^+)}\leq \infty\right\},$ where the $H_{\rm
TE}^s(\Gamma_R^+)$ norm is given by 
\[
\|u\|^2_{H^s_{\rm TE}(\Gamma_R^+)}=\left(\sum\limits_{n=0}^{\infty}(1+n^2)^s
|a_n|^2\right)^{1/2}.
\]

Following \cite{Li-JCM-2018}, we introduce the DtN operator
\begin{equation}\label{TE_DtN}
B_{\rm TE}u(R,\phi)=\kappa_0\sum\limits_{n=0}^{\infty}
\frac{H_n^{(1)'}(\kappa_0 R)}{H_n^{(1)}(\kappa_0 R )}a_n\cos n\phi.
\end{equation}
It is shown in \cite{Li-JCM-2018} that $B_{\rm TE}:
H^{1/2}_{\rm TE}(\Gamma_R^+)\to H_{TE}^{-1/2}(\Gamma_R^+)$ is
continuous. The following TBC can be imposed for the TE polarized wave field:
\[
\partial_{\rho}u=B_{\rm TE} u+g\quad {\rm on} ~ \Gamma_R^+,
\]
where $g=\partial_{\rho} u^{\rm ref}-B_{\rm TE}u^{\rm ref}$. Substituting
\eqref{TE_ref} into \eqref{TE_DtN} and using the Wronskian identity, we get
explicitly 
\begin{eqnarray*}
 g= -\frac{4\,{\rm i}}{\pi R}\frac{1}{H_0^{(1)}(\kappa_0 R)}-\frac{8}{\pi R} 
 \sum\limits_{n=1}^{\infty}\frac{{\rm i}^{n+1}}{H_n^{(1)}(\kappa_0 R)}\cos n(\theta-\pi/2)\cos n\phi.
\end{eqnarray*}

Therefore, the open cavity scattering problem in TE polarization can be reduced
equivalently into the boundary value problem 
\begin{equation*}
\begin{cases}
\nabla\cdot\left(\kappa^{-2}\nabla u\right) +u=0\quad &
{\rm in} ~ \Omega,\\
\partial_{\nu}u=0 & \quad {\rm on} ~ S\cup \Gamma_g,\\
\partial_{\rho} u-B_{\rm TE} u=g\quad & {\rm on} ~ \Gamma_R^+.
\end{cases}
\end{equation*}
The corresponding variational formulation is to find $u\in H^1(\Omega)$ such
that
\begin{equation}\label{TE_v1}
a_{\rm TE} (u, v)=\kappa_0^{-2}\int_{\Gamma_R^+}
g\bar{v}\,{\rm d}s\quad \forall\, v\in H^1(\Omega), 
\end{equation}
where the sesquilinear form $a_{\rm TE}: H^1(\Omega)\times
H^1(\Omega)\to\mathbb C$ is defined as
\[
a_{\rm TE} (u, v)=\int_{\Omega} \kappa^{-2}\nabla u\cdot
\nabla\bar{v}\,{\rm d}x-\int_{\Omega}
u\bar{v}\,{\rm d}x-\kappa_0^{-2}\int_{\Gamma_R^+}B_{\rm
TE} u\,\bar{v}\,{\rm d}s.
\]

\begin{theorem}
The variational problem \eqref{TE_v1} has a unique solution $u\in
H^1(\Omega)$, which satisfies the estimate
\[
\|u\|_{H^1(\Omega)}\lesssim \|g\|_{H^{-1/2}(\Gamma^{+}_R)}.
\]
\end{theorem}

\subsection{Finite element approximation}

Denote by $\mathcal M_h$ a regular triangulation of $\Omega$. Let
$V_h\subset H^1(\Omega)$ be a conforming finite 
element space, i.e., 
\[
V_h:=\left\{v_h\in C(\overline{\Omega}): v_h|_T\in P_m(K)\quad
\forall\, T\in\mathcal M_h\right\}. 
\]
The DtN operator \eqref{TE_DtN} needs to be truncated into a sum
of finitely many terms
\[
B_{\rm TE}^{N} u=\kappa_0\sum\limits_{n=0}^{N} 
\frac{H_n^{(1)'}(\kappa_0 R)}{H_n^{(1)}(\kappa_0 R)} a_n \cos n\phi.
\] 
The finite element approximation of variational problem \eqref{TE_v1} reads as
follows: find $u_h\in V_h$ such that 
\begin{equation}\label{TE_v2}
a_{\rm TE}^N (u^h, v^h)=\kappa_0^{-2}\int_{\Gamma_R^+}
g\,\overline{v^h}\, {\rm d}s\quad \forall\, v^h\in V_h,
\end{equation}
where the sesquilinear form $a_{\rm TE}^N: V_h\times V_h\to\mathbb C$ is
defined by
\[
a_{\rm TE}^N (u^h, v^h)=
\int_{\Omega} \kappa^{-2}\nabla u^h\cdot \nabla \overline{v^h}\,{\rm
d}x-\int_{\Omega}u^h\,\overline{v^h}\,{\rm d}x
-\kappa_0^{-2}\int_{\Gamma_R^+}B_{\rm TE}^{N} u^h
\,\overline{v^h}\,{\rm d}s.
\]

Similarly, we may assume that the variational problem \eqref{TE_v2} has a
unique solution $u_h\in V_h$ when $h$ is sufficiently small and $N$ is
sufficiently large.

\subsection{A posteriori error analysis}

Denote the jump residual of the interior edges as
\[
J_e =-\left(\kappa^{-2}\nabla u^h|_{T_1}\cdot
\nu_1+\kappa^{-2}\nabla u^h|_{T_2}\cdot
\nu_2\right).
\]
For any boundary edge $e\subset \Gamma_R^+$, the jump residual is defined as
\[
J_e=2\kappa_0^{-2}(B_{\rm TE}^N u^h-\nabla u^h\cdot
\nu-g).
\]
For any triangle $T\in\mathcal M_h$, denote by $\eta_T$ the local error
estimator:
\[
\eta_T=h_T \|H_{\rm TE}
u^h\|_{L^2(T)}+\left(\frac{1}{2}\sum\limits_{e\in\partial T}h_e
\|J_e\|_{L^2(e)}^2\right)^{1/2}
\]
where $H_{\rm TE}$ is the generalized Helmholtz operator given by $H_{\rm TE}
u=\nabla\cdot(\kappa^{-2}\nabla u)+u$.

The following theorem is the main result for the TE polarization. It presents
the a posteriori error estimate between the solution of the open cavity
scattering problem and the finite element solution. 

\begin{theorem}
Let $u$ and $u_h$ be the solution of \eqref{TE_v1} and \eqref{TE_v2},
respectively. There exists a positive integer $N_0$ independent 
of $h$ such that for $N>N_0$, the following a posteriori error estimate holds:
\[
\|u-u^h\|_{H^1(\Omega)}\lesssim
\bigg(\sum\limits_{T\in M_h}\eta_T^2\bigg)^{1/2}
+\bigg[\Big(\frac{\hat{R}}{R}\Big)^N 
+\Big(\frac{e\kappa_0 R}{2N}\Big)^{2N+4}\bigg]\|u^{\rm
ref}\|_{H^1(\Omega)}.
\]
\end{theorem}

The algorithm of the finite element DtN method for the TE polarization is the
same as that for the TM polarization. 

\section{TBC matrices}\label{section:correction}

The stiffness matrix $A$ arising from the discrete problem  \eqref{TM_v2} or
\eqref{TE_v2} can be written as
\begin{equation*}
A=B-F,
\end{equation*}
where the matrix $B$ comes from 
\[
\int_{\Omega}(\nabla u\cdot\nabla
\bar{v}-\kappa^2 u \bar{v}){\rm d}x
\]
for the TM polarization or 
\[
\int_{\Omega}(\kappa^{-2}\nabla u\cdot\nabla
\bar{v}-u \bar{v}){\rm d}x
\]
for the TE polarization, while the matrix $F$ accounts for the TBC part.
Since the matrix $B$ can be constructed in a standard way of the finite
element method, we focus on building the TBC matrix $F$ in this section.

\subsection{TM polarization}

Let $\mathcal M_h$ be a regular triangulation of the computational domain
$\Omega$. Correspondingly, the upper semi-circle $\Gamma_R^+$ is approximated by
line segments denoted by $\widehat{x_1\cdots x_M}$, where $x_i=
R (\cos \phi_i, \sin \phi_i), \phi_i<\phi_{i+1}, i=1, \dots, M-1$.
Denote the distance between each pair of adjacent points by 
$l_i=|x_i-x_{i-1}|$. 

On the boundary $\widehat{x_1\cdots x_M}$, let the finite element solution be
given by 
\[
u(x)=\sum\limits_{i=1}^{M} u(x_i) L_i(x),
\]
where $L_i$ are the basis functions defined as
\begin{equation*}
L_i(x)=\left\{
\begin{aligned}
&\frac{(x_{i+1}-x)}{l_{i+1}}\cdot\frac{(x_{i+1}-x_i)}{
l_{i+1}} ,\quad &
x\in\widehat{x_i x_{i+1}}, \\
&\frac{(x-x_{i-1})}{l_{i}}\cdot\frac{(x_i-x_{i-1})}
{l_{i}},\quad &
x\in\widehat{x_{i-1} x_{i}}.	
\end{aligned}
\right.	
\end{equation*}
Since the DtN operator is defined on the upper semi-circle $\Gamma_R^+$, there
are no points to the right of $x_1$ or to the left of 
$x_M$. For the convenience of discussion, we extend the basis functions by zero
to the line segments $\widehat{x_0 x_1}$ and $\widehat{x_M x_{M+1}}$.

By a straight forward computation, the Fourier coefficients of $u$ can be
obtained as
\begin{eqnarray*}
u^{(n)}(R)=\frac{2}{\pi R}\sum\limits_{i=1}^{M} u(R, \phi_i)\bigg[
\Big(\frac{1}{6}l_i\sin(n\phi_{i-1})+\frac{1}{3}l_i\sin(n\phi_i)\Big)	\\
+\Big(\frac{1}{3}l_{i+1}\sin(n\phi_{i})+\frac{1}{6}l_{i+1}
\sin(n\phi_{i+1})\Big)\bigg].	
\end{eqnarray*}
Substituting the above equation into the truncated DtN operator \eqref{TM_tDtN}
yields 
\begin{eqnarray*}
B^N_{\rm TM} u (R, \phi) =\sum\limits_{n=1}^{N}
\alpha^{(n)}\sum_{i=1}^{M} \beta_i^{(n)} u(R, \phi_i)\sin(n\phi),	
\end{eqnarray*}	
where
\begin{eqnarray*}
&&\alpha^{(n)}=\frac{2}{\pi R}\kappa_0\frac{H_n^{(1)'}(\kappa_0
R)}{H_n^{(1)}(\kappa_0 R)},\\
&&\beta_i^{(n)}=\Big(\frac{1}{6}l_i\sin(n\phi_{i-1})+\frac{1}{3}
l_i\sin(n\phi_i)\Big)+\Big(\frac{1}{3}l_{i+1}\sin(n\phi_{i})+\frac{1}{6}
l_{i+1}\sin(n\phi_{i+1})\Big).
\end{eqnarray*}
Noting that the basis functions $L_i$ are real-valued , we can compute the
TBC matrix as follows 
\begin{eqnarray*}
\int_{\Gamma_R^+} (B^N_{\rm TM} u)L_j {\rm d}s &=&
\sum_{i=1}^{M}\left[ \sum\limits_{n=1}^{N}\alpha^{(n)}\beta_i^{(n)}
\int_{\Gamma_R^{+}} \sin(n\phi)L_j(x)\,{\rm d}s\right] u(R,
\phi_i)\approx \sum_{i=1}^{M} F_{ji} u(R, \phi_i),
\end{eqnarray*}
where 
\begin{eqnarray*}
F_{ji}= \sum\limits_{n=1}^{N}\alpha^{(n)}\beta_i^{(n)}\gamma_j^{(n)}
\end{eqnarray*}
and
\[
\gamma_j^{(n)}=\Big(\frac{1}{6}l_j\sin(n\phi_{j-1})+\frac{1}{3}
l_j\sin(n\phi_j)\Big)+\Big(\frac{1}{3}l_{j+1}\sin(n\phi_{j})+\frac{1}{6}
l_{j+1}\sin(n\phi_{j+1})\Big).
\]

\subsection{TE polarization}

In this section, we consider how to construct the matrix $B$ for the TE
polarization. Using the same basis functions as those for the TM polarization, 
we can expand the finite element solution $u$ on the boundary as
\[
u(x)=\sum\limits_{i=1}^{M} u(x_i) L_i(x).
\]
It follows from the straight forward calculation that the Fourier coefficients
of $u$ are 
\begin{eqnarray*}
u^{(0)}(R)& = &\frac{1}{\pi R} \sum\limits_{i=1}^{N} u(R,
\theta_i)\Big(\frac{1}{2}l_i+\frac{1}{2}l_{i+1}\Big),\\
u^{(n)}(R)&=&\frac{2}{\pi R}\sum\limits_{i=1}^{N} u(R, \phi_i)\bigg[
\Big(\frac{1}{6} l_i\cos(n\phi_{i-1})+\frac{1}{3}l_i\cos(n\phi_i)\Big)\\
&&+\Big(\frac{1}{3}l_{i+1}\cos(n\phi_{i})+\frac{1}{6}l_{i+1}\cos(n\phi_{i+1}
)\Big)\bigg],\quad n\geq 1. 
\end{eqnarray*}
Substituting the Fourier coefficients into the truncated DtN operator $B^N_{\rm
TE}$, we get
\[
B^N_{\rm TE} u=\sum\limits_{n=0}^{N} \alpha^{(n)}\sum\limits_{i=1}^{M}
\beta_i^{(n)} u(R, \phi_i) \cos(n\phi),
\]
where
\begin{equation*}
\alpha^{(0)} =\frac{\kappa_0}{\pi R}\frac{H_0^{(1)'}(\kappa_0
R)}{H_0^{(1)}(\kappa_0 R)},\quad
		 \alpha^{(n)} =\frac{2\kappa_0}{\pi R}\frac{H_n^{(1)'}(\kappa_0
R)}{H_n^{(1)}(\kappa_0 R)}, \quad n\geq 1, 
\end{equation*}
and
\begin{align*}
\beta_i^{(0)}& =\frac{1}{2} l_i+\frac{1}{2} l_{i+1}, \\
\beta_i^{(n)} & =\Big(\frac{1}{6} l_i
\cos(n\phi_{i-1})+\frac{1}{3}l_i\cos(n\phi_i)\Big)
+\Big(\frac{1}{3} l_{i+1}
\cos(n\phi_{i})+\frac{1}{6}l_{i+1}\cos(n\phi_{i+1})\Big),\quad n\geq 1.
\end{align*}
Then the TBC matrix can be obtained as follows
\[
F_{ji}=\sum\limits_{n=0}^{N}\alpha^{(n)}\beta_i^{(n)}\int_{
\Gamma_R^+} \cos(n\phi)L_j\,{\rm
d}s\approx\sum\limits_{n=0}^{N}\alpha^{(n)}\beta_i^{(n)}\gamma_j^{(n)},
\]
where
\begin{align*}
\gamma_j^{(0)}&=\frac{1}{2} l_j+\frac{1}{2} l_{j+1}, \\
 \gamma_j^{(n)}&=\Big(\frac{1}{6} l_j
\cos(n\phi_{j-1})+\frac{1}{3}l_j\cos(n\phi_j)\Big)
+\Big(\frac{1}{3} l_{j+1}
\cos(n\phi_{j})+\frac{1}{6}l_{j+1}\cos(n\phi_{j+1})\Big),\quad n\geq 1.
\end{align*}

\section{Radar cross section}\label{section:rcs}

The physical parameter of interest is the radar cross section (RCS), which
measures the detectability of a target by a radar system. In two dimensions,
the RCS is defined by
\[
\sigma(\varphi):=\lim\limits_{r\rightarrow\infty}2\pi r\frac{|u^{\rm s}(r,
\varphi)|^2}{|u^{\rm i}|^2},
\]
where $u^{\rm i}$ and $u^{\rm s}$ are the incident and scattered fields,
respectively, $\varphi$ is the observation angle. Since the incident wave is
chosen to be a plane wave, the
RCS reduces to 
\[
\sigma(\varphi)=\lim\limits_{r\rightarrow\infty}2\pi r |u^{\rm s}(r,
\varphi)|^2.
\]
When the incident angle $\theta$ and the observation angle $\varphi$ are the
same, $\sigma$ is called the backscatter RCS, which is defined by
\[
 \text{Backscatter RCS}(\sigma)(\varphi)=10\log\sigma(\varphi) {\rm dB}. 
\]
In this section, we derive the formulas of the backscatter RCS when the
scattered field is measured on the aperture $\Gamma$ and on the upper
semi-circle $\Gamma_R^+$, respectively.

\subsection{TM polarization with measurement on $\Gamma$}

Let $x=(x_1, x_2)$ and $y=(y_1, y_2)$ be the observation point and the source
point, respectively. The fundamental solution for the two dimensional Helmholtz
equation is defined by
\[
\Phi(x, y)=\frac{\rm i}{4}H_0^{(1)}(\kappa
|x-y|),
\]
where $H_0^{(1)}$ is the Hankel function of the first kind with order $0$. 

In the TM polarization, the half space Green's function is
\[
G_{\rm TM}(x, y)=\Phi(x, y)-\Phi(\tilde{x}, y),
\]
where $\tilde{x}=(x_1, -x_2)$ is the reflection point of $x$ with respect to
the $x_1$-axis. Let  $x=r(\cos\varphi, \sin\varphi)$. We have from the
Helmholtz equation and the Green theorem of the second kind that 
\begin{eqnarray*}
-u^{s}(x) &=& \int_{\Omega}\left(\Delta G_{\rm TM}(x, y) +\kappa^2
G_{\rm TM}(x, y)\right)u^{\rm s}(y){\rm d}y\\
&=& \frac{e^{{\rm i}\,\kappa r}}{\sqrt{r}}\frac{e^{{\rm
i}\,\frac{\pi}{4}}}{\sqrt{8\pi k}}
\int_{\Gamma} u^s(y_1, 0) 2{\rm i}\kappa\sin\varphi e^{{\rm
i}\kappa\cos\varphi y_1}{\rm d}y_1+o(r^{-1/2}),
\end{eqnarray*}
where $\Gamma$ is the aperture of the cavity. Substituting the above equation
into the definition of $\sigma$, we have 
\begin{eqnarray}\label{RCS_TM}
\sigma(\varphi) =\kappa \left| \sin\varphi\int_{\Gamma} u^{\rm s}(y_1, 0)
e^{{\rm i}\kappa \cos\varphi y_1}{\rm d}y_1\right|^2.
\end{eqnarray}
	
\subsection{TE polarization  with measurement on $\Gamma$}

In the TE polarization, the half space Green's function is
\[
G_{\rm TE}(x, y)=\Phi(x, y)+\Phi(\tilde x, y).
\]
By Green's theorem, we may similarly obtain 
\begin{eqnarray*}
-u^{s}(x) = 2\frac{e^{{\rm i}\kappa r}}{\sqrt{r}}\frac{e^{{\rm
i}\frac{\pi}{4}}}{\sqrt{8\pi k}}
\int_{\Gamma} \frac{\partial}{\partial y_2}u^{\rm s}(y_1, 0) e^{{\rm
i}\kappa \cos\varphi y_1}{\rm d}y_1+o(r^{-1/2}), 
\end{eqnarray*}
which gives
\begin{eqnarray}\label{RCS_TE}
\sigma(\varphi)= \frac{1}{\kappa}\left| \int_{\Gamma} \frac{\partial}{\partial
y_2}u^{\rm s}(y_1, 0) e^{{\rm i}\kappa\cos\varphi y_1}{\rm d}y_1\right|^2.
\end{eqnarray}

\subsection{TM polarization with measurement on $\Gamma_R^+$}

Using the half space Green's function, we may obtain 
\begin{eqnarray*}
u^{\rm s}(x) &=& -\int_{\Gamma_R^+} G_{\rm TM}(x, y) \frac{\partial}{\partial
\nu}u^{\rm s}(y){\rm d}s_y+\int_{\Gamma_R^+} \frac{\partial}{\partial
\nu(y)}G_{\rm TM}(x, y)u^{\rm s}(y){\rm d}s_y\\
&=& \frac{e^{{\rm i}\kappa r}}{\sqrt{r}}\frac{e^{{\rm
i}\frac{\pi}{4}}}{\sqrt{8\pi\kappa}}R\bigg\{
\int_{0}^{\pi}\big[{\rm i}\kappa \cos(\phi-\varphi) u^{\rm s}(R,
\phi)-\frac{\partial}{\partial r}u^{\rm s}(R, \phi)\big]
e^{{\rm i}\kappa R\cos\left(\phi-\varphi\right)}{\rm
d}\phi\\
&&\quad -\int_{0}^{\pi}\Big[{\rm i}\kappa \cos(\phi+\varphi)
u^{\rm s}(R, \phi)-\frac{\partial}{\partial r}u^{\rm s}(R, \phi)\Big]
e^{{\rm i}\kappa R\cos\left(\phi+\varphi\right)}{\rm d}\phi
\bigg\}+o(r^{-1/2}).
\end{eqnarray*}
Substituting the above equation into $\sigma$, we get
\begin{eqnarray}\label{RCS_TMf}
\sigma(\varphi) =\frac{R^2}{4\kappa}\Bigg|
\int_{0}^{\pi}\Big[{\rm i}\kappa \cos(\phi-\varphi) u^{\rm s}(R,
\phi)-\frac{\partial}{\partial r}u^{\rm s}(R, \phi)\Big]
e^{{\rm i}\kappa R\cos\left(\phi-\varphi\right)}{\rm d}\phi
\notag\\
 -\int_{0}^{\pi}\Big[{\rm i}\kappa \cos(\phi+\varphi) u^{\rm s}(R,
\phi)-\frac{\partial}{\partial r}u^{\rm s}(R, \phi)\Big]
e^{{\rm i}\kappa R\cos\left(\phi+\varphi\right)}{\rm d}\phi
\Bigg|^2.
\end{eqnarray}

\subsection{TE polarization with measurement on $\Gamma_R^+$}

Using the half space Green's function for the TE polarization, we may
similarly obtain 
\begin{eqnarray*}
u^s(x) &=& -\int_{\Gamma_R^+} G_{\rm TE}(x, y)
\frac{\partial}{\partial \nu}u^{\rm s}(y){\rm d}s_y
+\int_{\Gamma_R^+} \frac{\partial}{\partial
\nu(y)}G_{\rm TE}(x, y)u^{\rm s}(y){\rm d}s_y\\
&=& \frac{e^{{\rm i}\kappa r}}{\sqrt{r}}\frac{e^{{\rm
i}\frac{\pi}{4}}}{\sqrt{8\pi\kappa}}R\Big\{
\int_{0}^{\pi}\left[{\rm i}\kappa \cos(\phi-\varphi) u^{\rm s}(R,
\phi)-\frac{\partial}{\partial r}u^{\rm s}(R, \phi)\right]
e^{{\rm i}\kappa R\cos\left(\phi-\varphi\right)}{\rm
d}\phi\\
&&\quad +\int_{0}^{\pi}\left[{\rm i}\kappa \cos(\phi+\varphi) u^{\rm s}(R,
\phi)-\frac{\partial}{\partial r}u^{\rm s}(R, \phi)\right]
e^{{\rm i}\kappa R\cos\left(\phi+\varphi\right)}{\rm d}\phi
\Big\}+o(r^{-1/2}).
\end{eqnarray*}	
Substituting the above scattered field into $\sigma$ gives 
\begin{eqnarray}\label{RCS_TEf}
\sigma(\varphi) &=&\frac{R^2}{4\kappa}\Bigg|
\int_{0}^{\pi}\Big[{\rm i}\kappa \cos(\phi-\varphi) u^{\rm s}(R,
\phi)-\frac{\partial}{\partial r}u^{\rm s}(R, \phi)\Big]
e^{{\rm i}\kappa R\cos\left(\phi-\varphi\right)}{\rm d}\phi
\notag\\
&&\quad+\int_{0}^{\pi}\Big[{\rm i}\kappa \cos(\phi+\varphi) u^{\rm s}(R,
\phi)-\frac{\partial}{\partial r}u^{\rm s}(R, \phi)\Big]
e^{{\rm i}\kappa R\cos\left(\phi+\varphi\right)}{\rm
d}\phi\Bigg|^2. 
\end{eqnarray}

\section{Numerical experiments}\label{section:numerical}

In this section, we present some examples to demonstrate the numerical
performance of the proposed method. All the following experiments are done
by using FreeFem \cite{H-JNM-2012}.

\subsection{Example 1}

This is a benchmark example which is frequently used to test the numerical
solutions \cite{j-02}. We consider the TM polarized wave fields. The cavity is a
rectangle with width $\lambda$ and depth $0.25\lambda$. The geometry of the
cavity is shown in Figure \ref{TM_Example1}. The wavenumber in the free space is
$\kappa_0=32\pi$ and wavelength $\lambda=2\pi/\kappa_0=1/16$. The TBC is imposed
on the semi-circle with radius a half wavelength. The TBC truncation number
$N=20$. We consider two cases: an empty cavity and a cavity filled with a lossy
medium with the electric permittivity $\epsilon=4+{\rm i}$ and the magnetic
permeability $\mu=1$. First, we compute the backscatter RCS based on
\eqref{RCS_TMf} by using the adaptive finite element DtN method. The adaptive
mesh refinement is stopped when the total number of nodal points is over 15000.
The backscatter RCS is shown as solid lines for both cases in Figure
\ref{TM_Example1}. To make a comparison, we also compute the backscatter RCS
based on \eqref{RCS_TM} by using the coupling of finite element method and
boundary integral method (FEMBIM) proposed in \cite{LA-JCP-13}. The compared
result is obtained by using a uniform mesh with the total number of nodal points
$101105$. Clearly, we can get the same accuracy but with a relatively small
number of nodal points by applying the adaptive finite element DtN method. Using
the incident angle $\theta=\pi/3$ as a representative example, we present the
refined mesh after 4 iterations with a total number of nodal points $1674$ in
Figure \ref{TM_Example1_2}. As expected, the mesh is refined locally near the
two corners of the cavity. The a posteriori error estimates are plotted in
Figure \ref{TM_Example1_2} to show the convergence rate of the method. It
indicates that the meshes and the associated numerical complexity are
quasi-optimal, i.e., $\varepsilon_h=O({\rm DoF}_h^{-1/2})$ holds asymptotically,
where ${\rm DoF}_h$ is the degree of freedom or the number of nodal points for
the mesh $\mathcal M_h$. Using the same incident angle $\theta=\pi/3$, we
compare the backscatter RCS by using the finite element DtN method with the
adaptive mesh and uniform mesh refinements in Table \ref{ex1_tab}. Clearly, it
shows the advantage of using adaptive mesh refinements since it may give more
accurate results by using fewer number of nodal points. 
 
\begin{figure}
\centering
\includegraphics[width=0.45\textwidth]{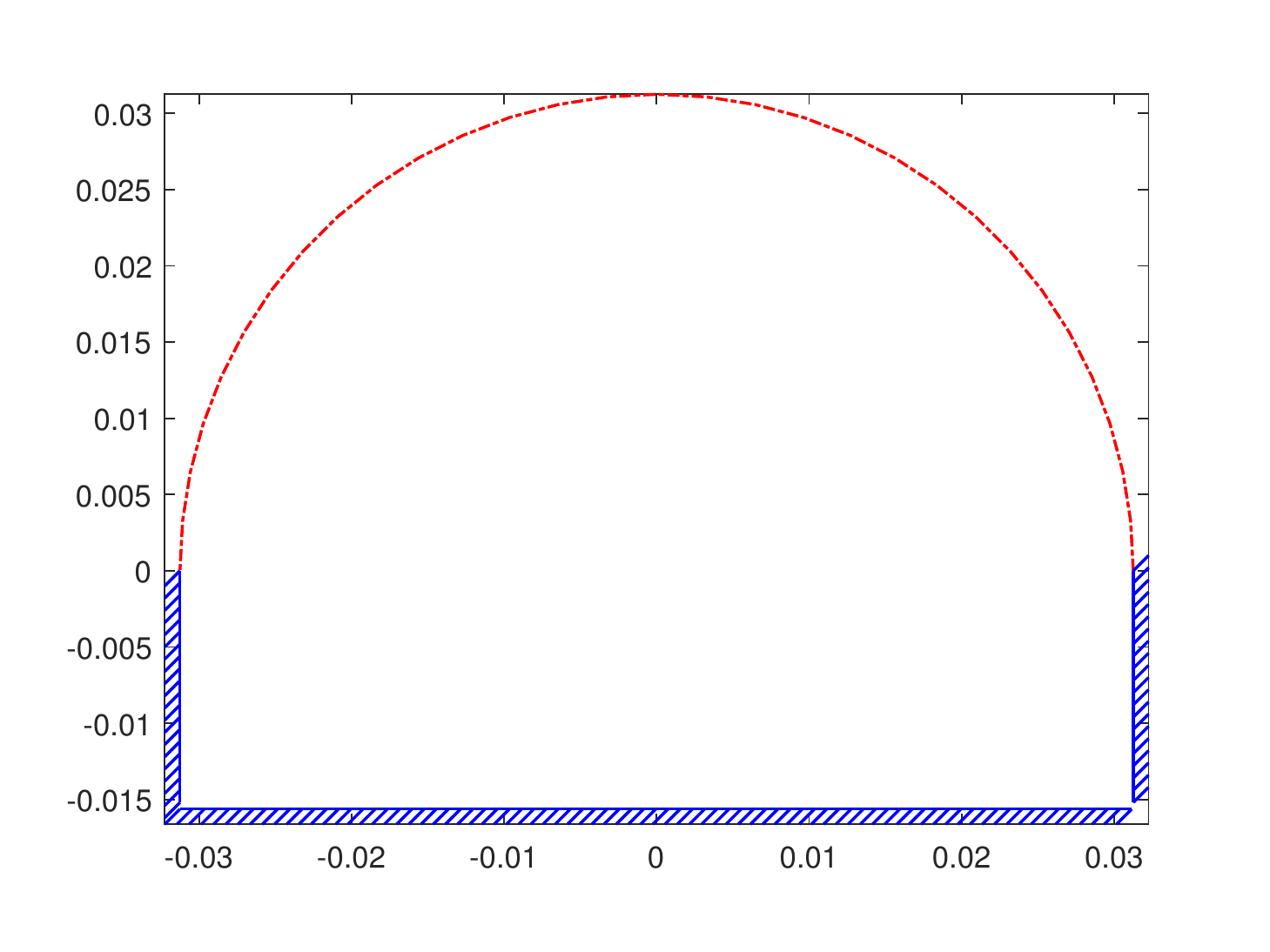}\quad
\includegraphics[width=0.45\textwidth]{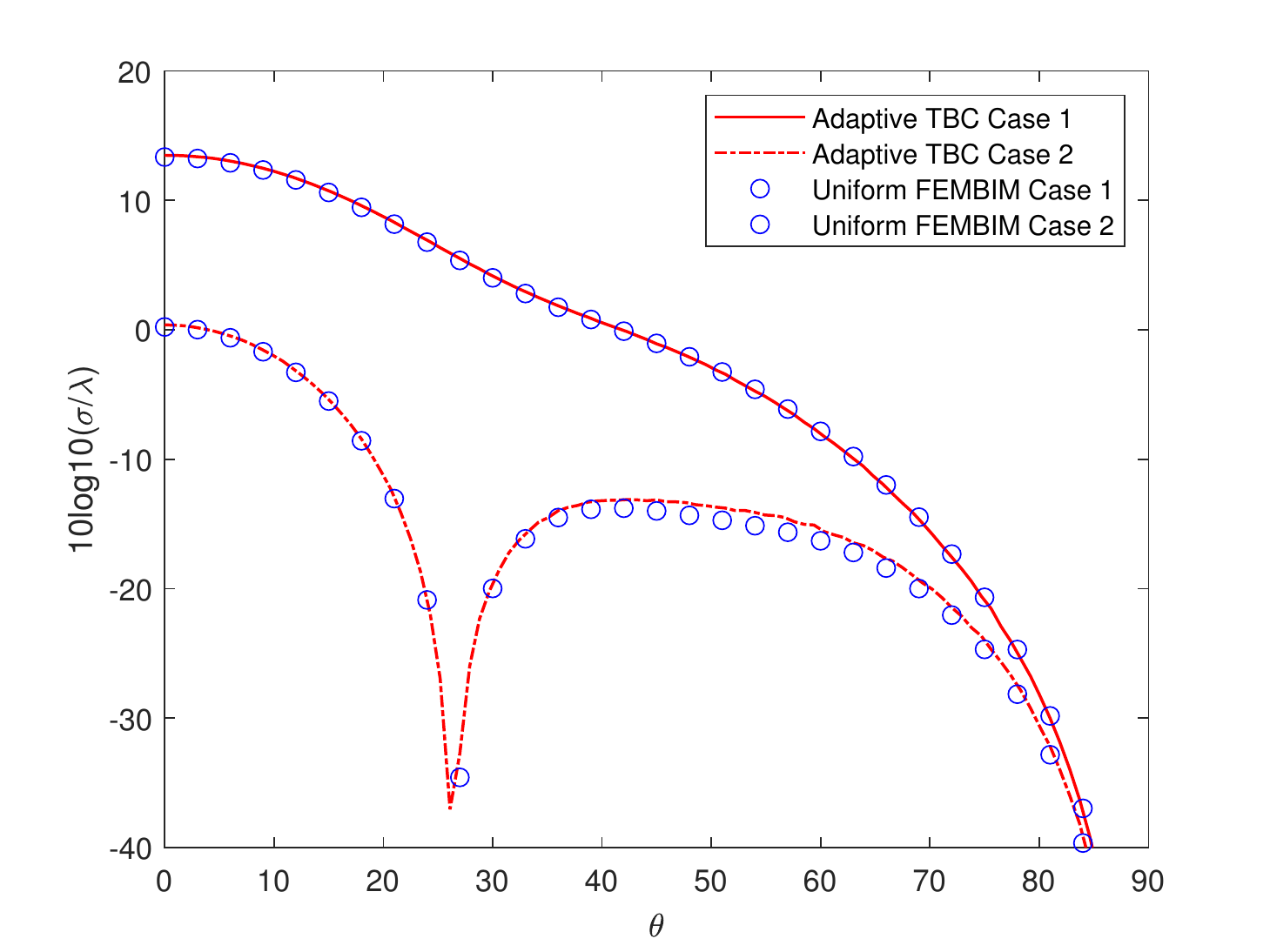}
\caption{Example 1: (left) Geometry of the cavity; (right) Backscatter RCS for
both cases by using the adaptive finite element DtN method (Adaptive TBC) and
the coupling of finite element method and boundary integral method (FEMBIM).}
\label{TM_Example1}
\end{figure}

\begin{figure}
\centering
\includegraphics[width=0.45\textwidth]{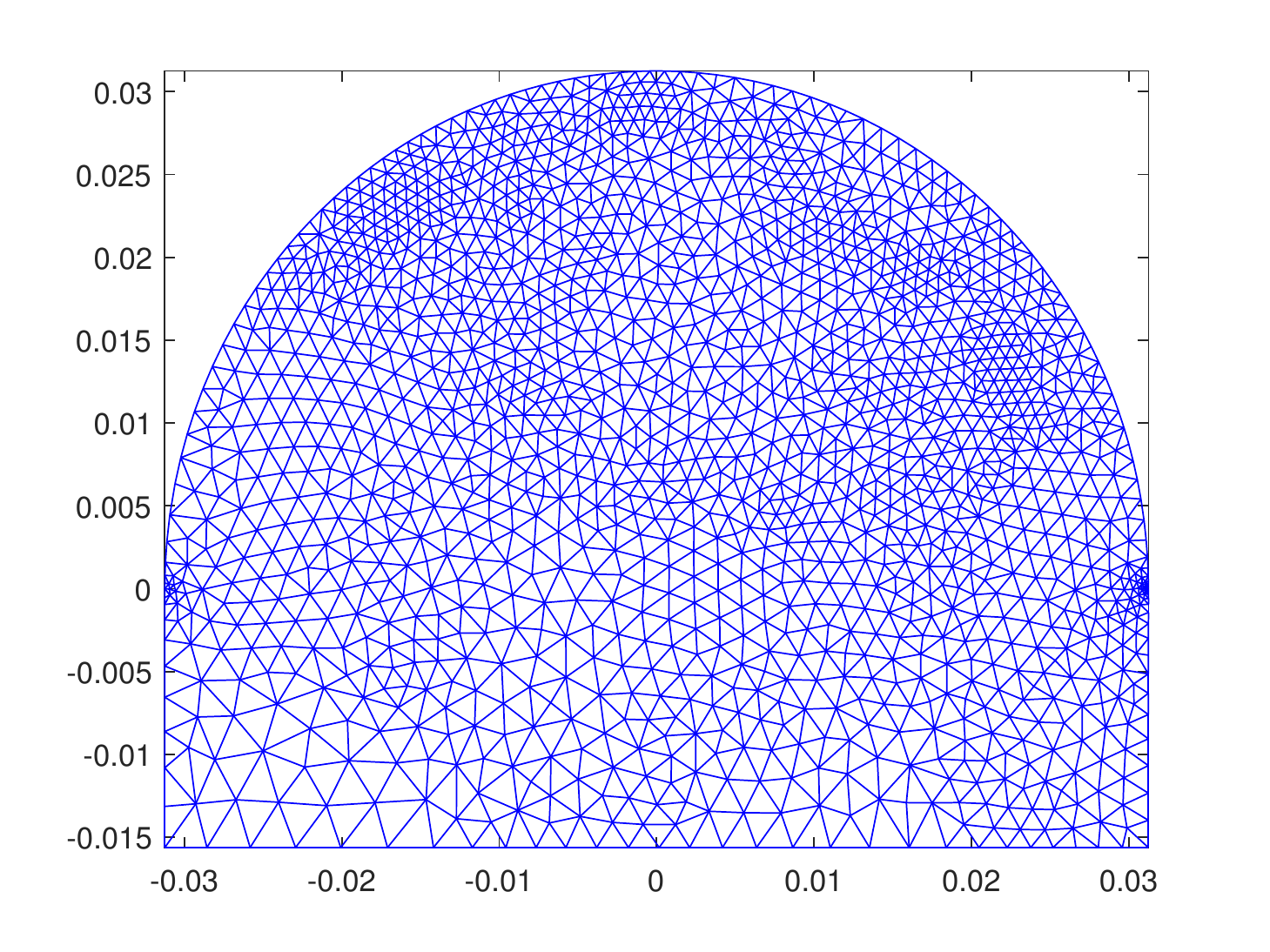}\quad
\includegraphics[width=0.45\textwidth]{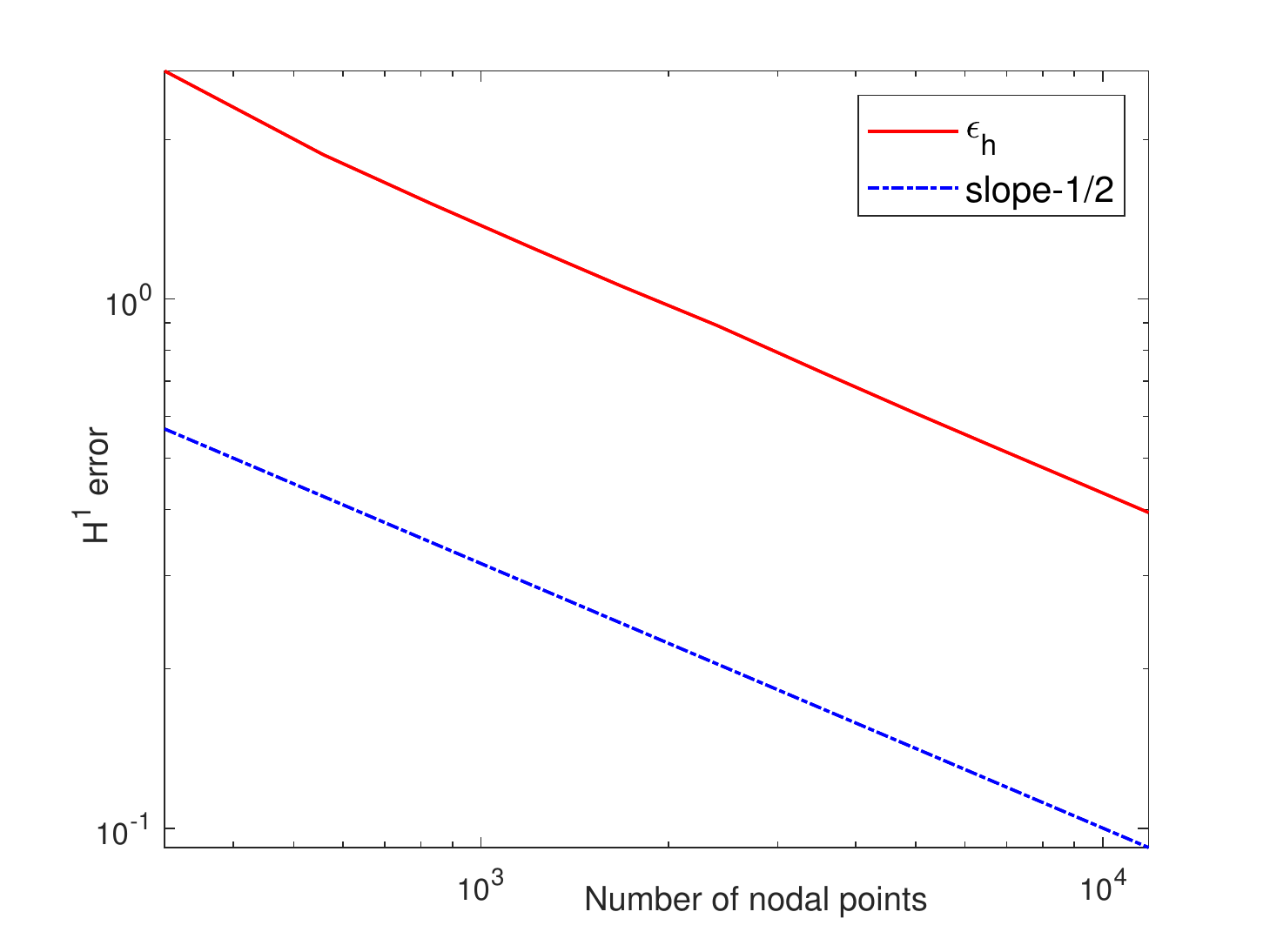}
\caption{Example 1: (left) Adaptive mesh after 4 iterations with a total number
of nodal points $1674$; (right) Quasi-optimality of the a posteriori error
estimates.}
\label{TM_Example1_2}
\end{figure}

\begin{table}
\caption{Example 1: comparison of numerical results using the adaptive meshe
and uniform mesh refinements. ${\rm DoF}_h$ is the degree of freedom or the
number of nodal points for the mesh $\mathcal M_h$.}
\begin{center}
\begin{tabular}{ |p{1.5cm}|p{2.5cm}|p{1.5cm}|p{2.5cm}|  }
 \hline
 \multicolumn{2}{|c}{Adaptive Mesh}
 &  \multicolumn{2}{c|}{Uniform Mesh}\\
 \hline
 ${\rm DoF}_h$ & RCS & ${\rm DoF}_h$ & RCS \\
 \hline
 310    & 0.0061367 & 310      & 0.0061367\\
 509    & 0.0042423 & &\\
 1173   & 0.0029691 & 1199	& 0.0032574\\
 2298   & 0.0023687 & &\\
 4438   & 0.0021304 & 4648	& 0.0022952\\
 10293  & 0.0018429 & &\\
 17875  & 0.0017774 & 18146	& 0.0018434\\
 23463  & 0.0017182 & &\\
 39878  & 0.0016583 & 40234	& 0.0017165\\
 65544  & 0.0016212 & &\\
 104706 & 0.0015878 & 111274	& 0.0016248\\
 \hline
\end{tabular}
\end{center}
\label{ex1_tab}
\end{table}

\subsection{Example 2}

This example also concerns the TM polarization. We compute the backscatter RCS
for a coated rectangular cavity, which has a width $1.2 \lambda$ and a depth
$0.8 \lambda$. The each vertical side of the cavity wall is coated with a thin
layer of some absorbing material, as seen in Figure \ref{TM_E2_RCS}. The
thickness of the coating is $0.012\lambda$ for both sides. The coating is made
of a homogeneous absorbing medium, which has a relative permittivity
$\epsilon_{\rm r}=12+0.144 {\rm i}$ and a relative permeability
$\mu_r=1.74+3.306{\rm i}$. This is a multi-scale problem and it is very
difficult to compute the numerical solution by using the finit element with
uniform mesh refinements in order to resolve the thin absorbing layers. Figure
\ref{TM_E2_RCS} plots the backscatter RCS. Again, the solid line is the result
by using the adaptive finite element DtN method, while the circles stand for the
result by using the coupling of the finite element method and the boundary
integral method (FEMBIM). We take the same stopping strategy as the one for
Example 1: the adaptive finite element DtN method is stopped when the number of
nodal points is over 15000. For the FEMBIM, the result is computed by using a
uniform mesh with the number of nodal points $112059$. Using a representative
example of incident angle  $\theta=\pi/3$, we present the refined mesh
after 2 iterations with 1508 ${\rm DoF}_h$ and the a posteriori error
estimates in Figure \ref{TM_E2_ref}. It is clear to note that the method can
capture the behavior of the numerical solution in the two thin absorbing
layers and displays the quasi-optimality between the meshes and the associated
numerical complexity, i.e., $\varepsilon_h=O({\rm DoF}_h^{-1/2})$ holds
asymptotically. As a comparison, we show the backscatter RCS by using the
finite element DtN method with the adaptive mesh and uniform mesh refinements in
Table \ref{ex2_tab}. Apparently, the adaptive mesh refinements yields a better
numerical performance than the uniform mesh refinements does, since the former
can give more accurate results even by using fewer number of nodal points.

\begin{figure}
\centering
\includegraphics[width=0.45\textwidth]{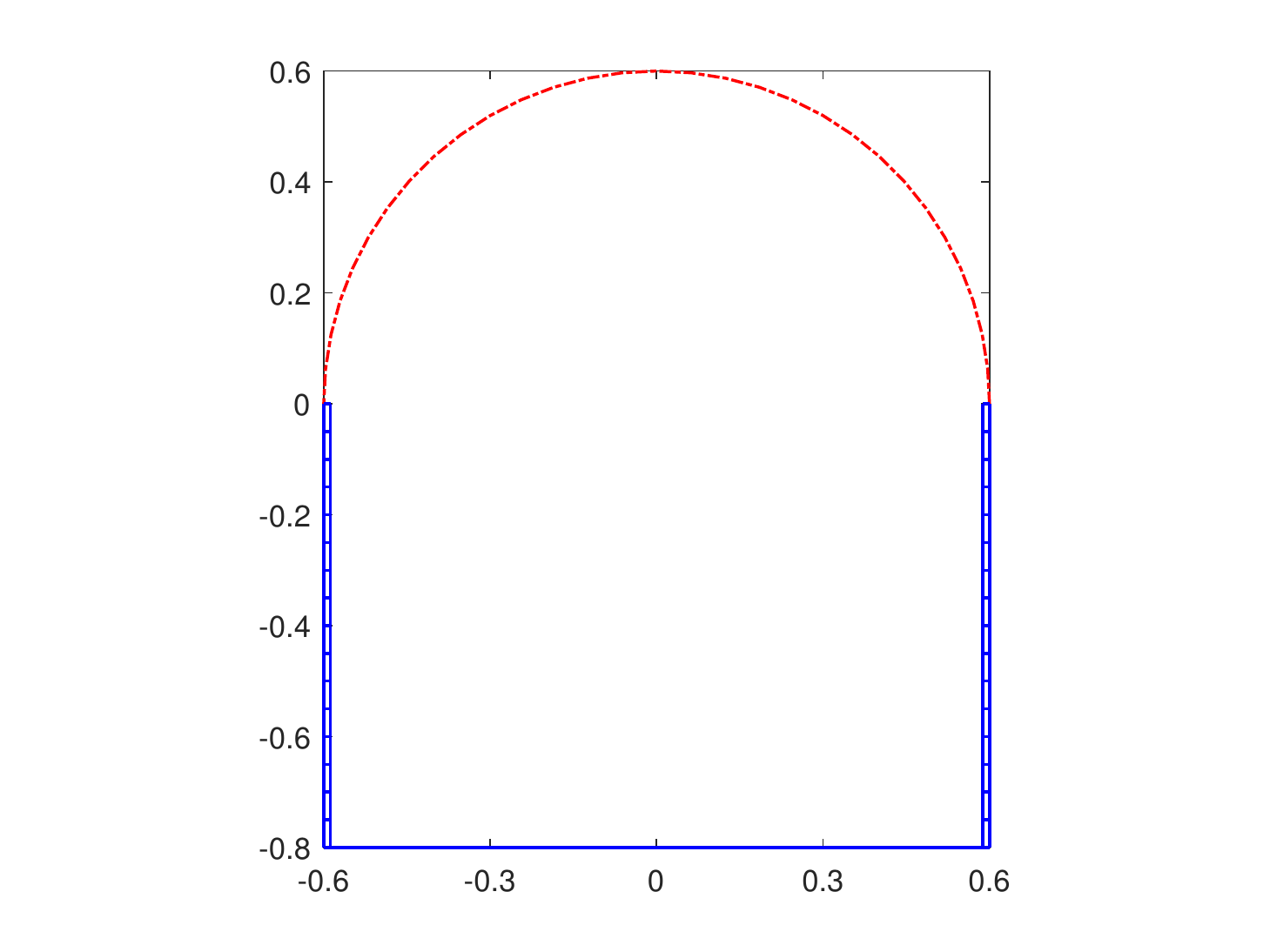}\quad
\includegraphics[width=0.45\textwidth]{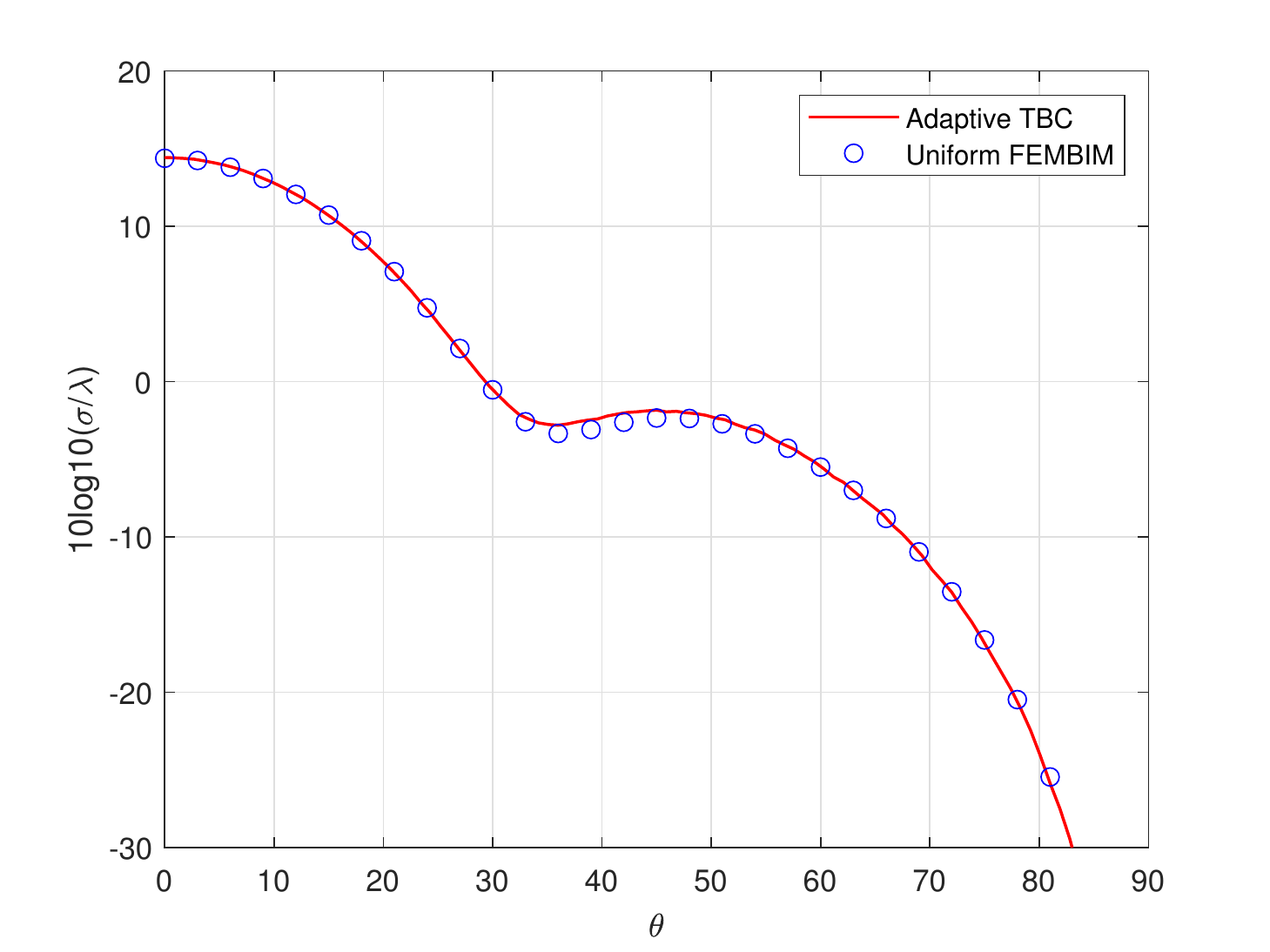}
\caption{Example 2: (left) Geometry of the cavity with thin absorbing layers on
the two vertical sides; (right) Backscatter RCS by using the adaptive finite
element DtN method (Adaptive TBC) and the coupling of finite element method and
boundary
integral method (FEMBIM).}
\label{TM_E2_RCS}
\end{figure}

\begin{figure}	
\centering
\includegraphics[width=0.45\textwidth]{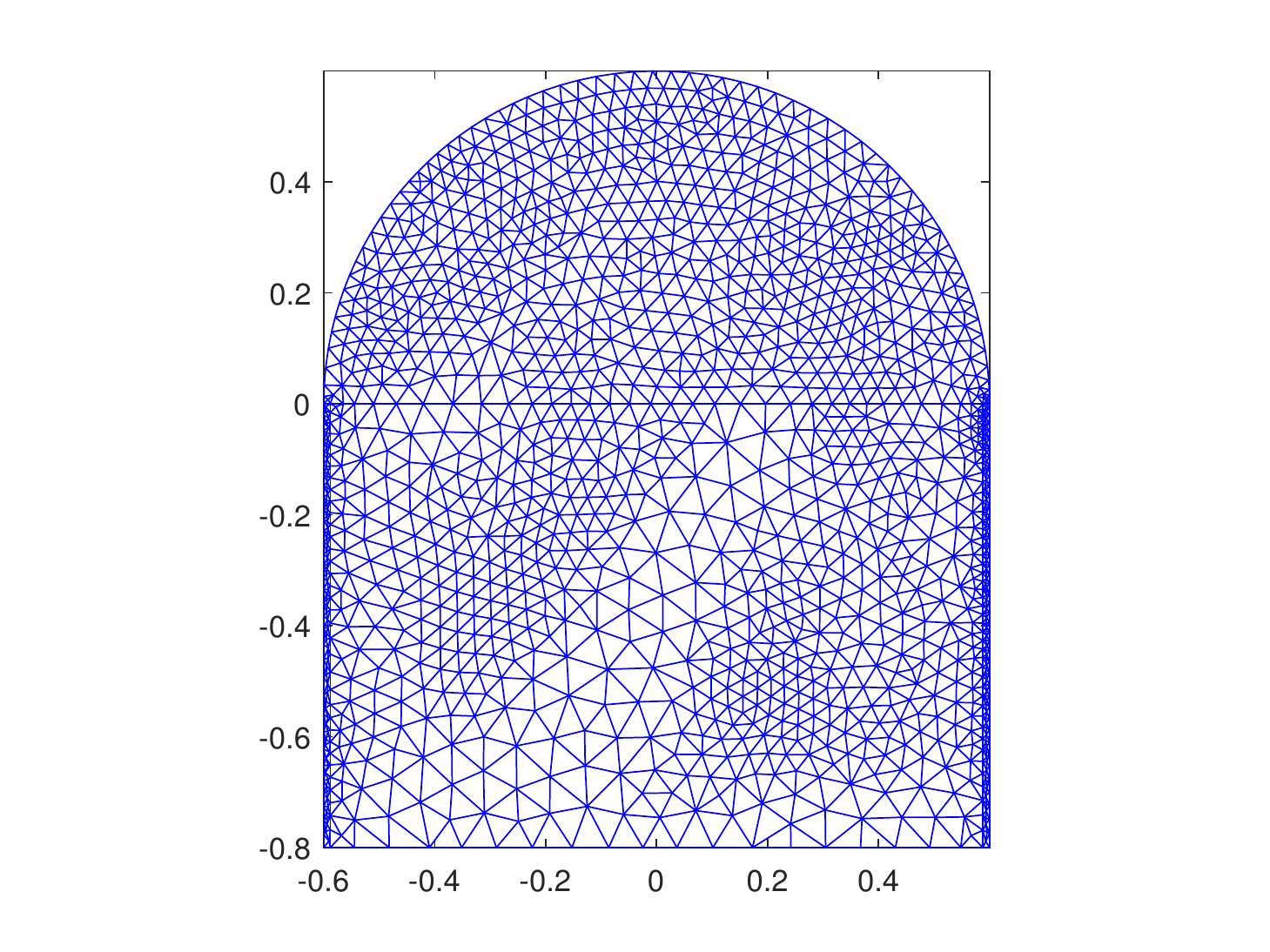}\quad
\includegraphics[width=0.45\textwidth]{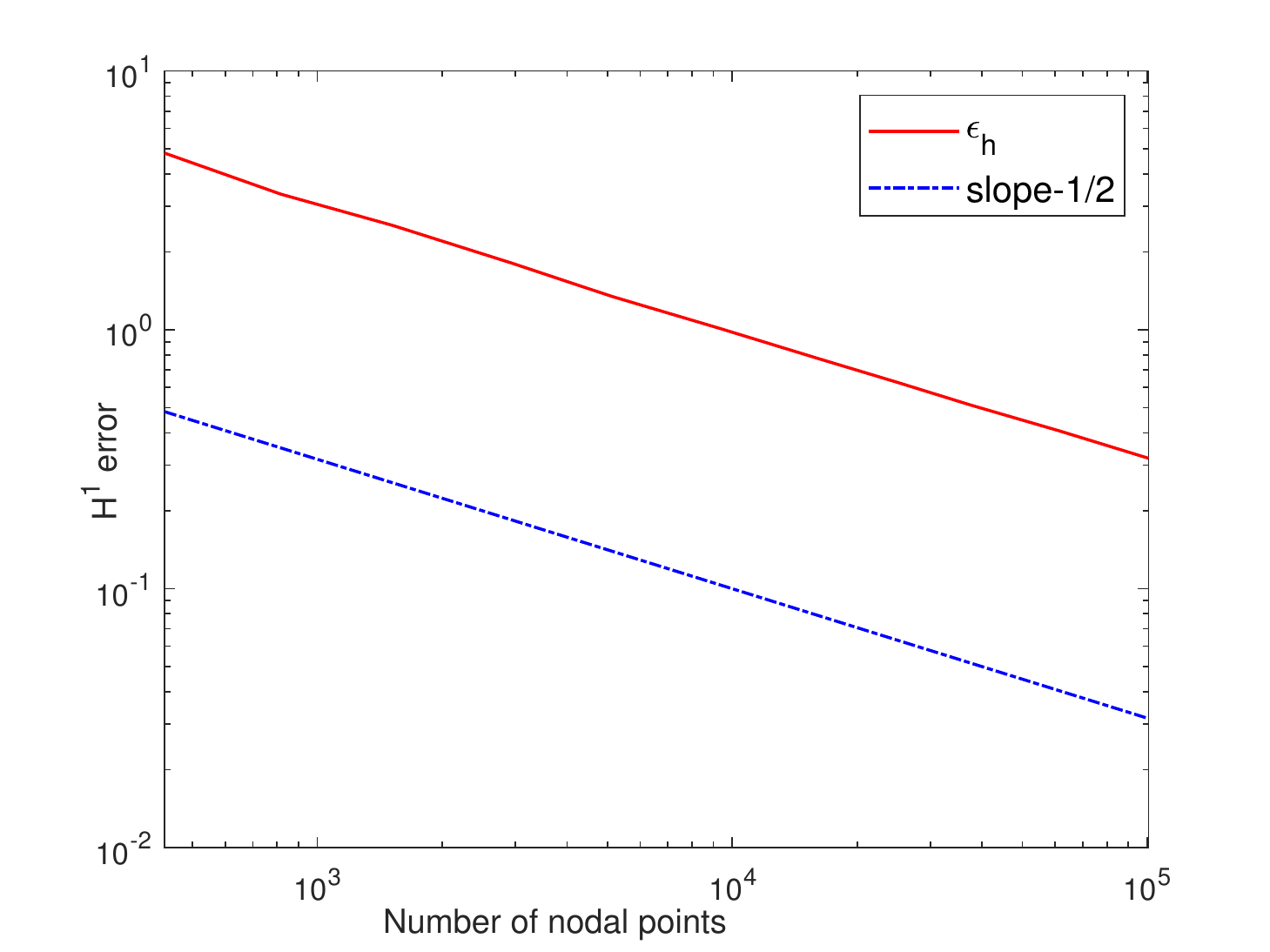}
\caption{Example 2: (left) Adaptive mesh after 2 iterations with a total number
of nodal points $1508$; (right) Quasi-optimality of the a posteriori error
estimates.}
\label{TM_E2_ref}
\end{figure}

\begin{table}
\caption{Example 2: comparison of numerical results using the adaptive meshe
and uniform mesh refinements. ${\rm DoF}_h$ is the degree of freedom or the
number of nodal points for the mesh $\mathcal M_h$.}
\begin{center}
\begin{tabular}{|p{1.5cm}|p{2.5cm}|p{1.5cm}|p{2.5cm}| }
 \hline
 \multicolumn{2}{|c}{Adaptive mesh}
 &  \multicolumn{2}{c|}{Uniform mesh}\\
 \hline
${\rm DoF}_h$ & RCS & ${\rm DoF}_h$ & RCS \\
\hline
429     & 0.84613846   & 429    & 0.8461385\\
818     & 0.54563374   &        & \\
1518    & 0.61134157   &        & \\
2924    & 0.55103383   & 3563   & 0.6429673\\
5118    & 0.57427862   &        & \\
9358    & 0.57726558   &        & \\
15928   & 0.58262461   & 13876  & 0.6048405\\
24786   & 0.58383447   &        & \\
37473   & 0.57827731 & 30981    & 0.5982348\\
61332   & 0.57586437 & 55028    & 0.5967787\\
100740  & 0.57570343 & 124138   & 0.5925520\\
161478  & 0.57704918   & 277710 & 0.5908951 \\
 \hline
\end{tabular}
\end{center}
\label{ex2_tab}
\end{table}

\subsection{Example 3}

In the above two examples, the rectangle-shaped cavities are below the ground.
For such cavities, we may either use the coupling of the finite element method
and boundary integral method (FEMBIM) method \cite{LA-JCP-13} or the finite
element perfectly matched layer (FEPML) method \cite{XW-CICP-2016} to solve the
scattering problems. In this example, we consider the TM case but the structure
of the cavity is above the ground. The width and depth of the cavity
is $1.2\lambda$ and $0.8\lambda$, respectively. However, we set two thin
rectangular PEC humps in the middle of the cavity with height
$\frac{16}{15}\lambda$ and $\frac{8}{15}\lambda$, respectively. The width is
$\frac{1}{20}\lambda$ for both humps. The geometry of the cavity and the
backscatter RCS are shown in Figure \ref{TM_E3_G}. Again, the stopping criterion
is that the
mesh refinement is stopped when the number of nodal points is over 15000. Using
the incident angle $\theta=\pi/3$ as an example, we show the refined mesh after
two iterations with the number of nodal points $1409$ and the a posteriori error
estimates in Figure \ref{TM_E3_r}. The adaptive DtN method is able to
generate locally refined meshes around the corners of the cavity where the
solution has a singularity. The quasi-optimality is also obtained for the a
posteriori error estimates. 

\begin{figure}
\centering
\includegraphics[width=0.45\textwidth]{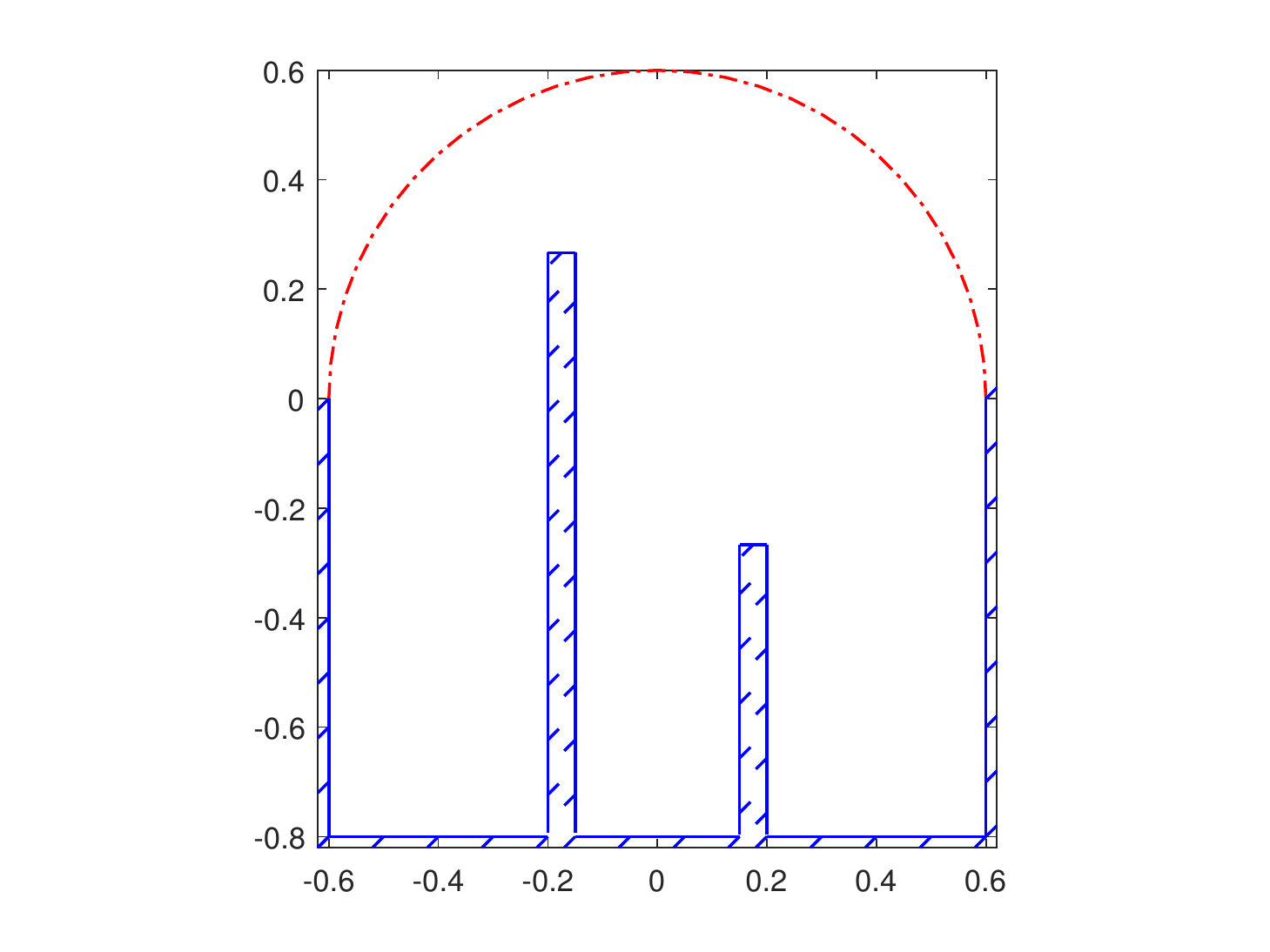} \quad
\includegraphics[width=0.45\textwidth]{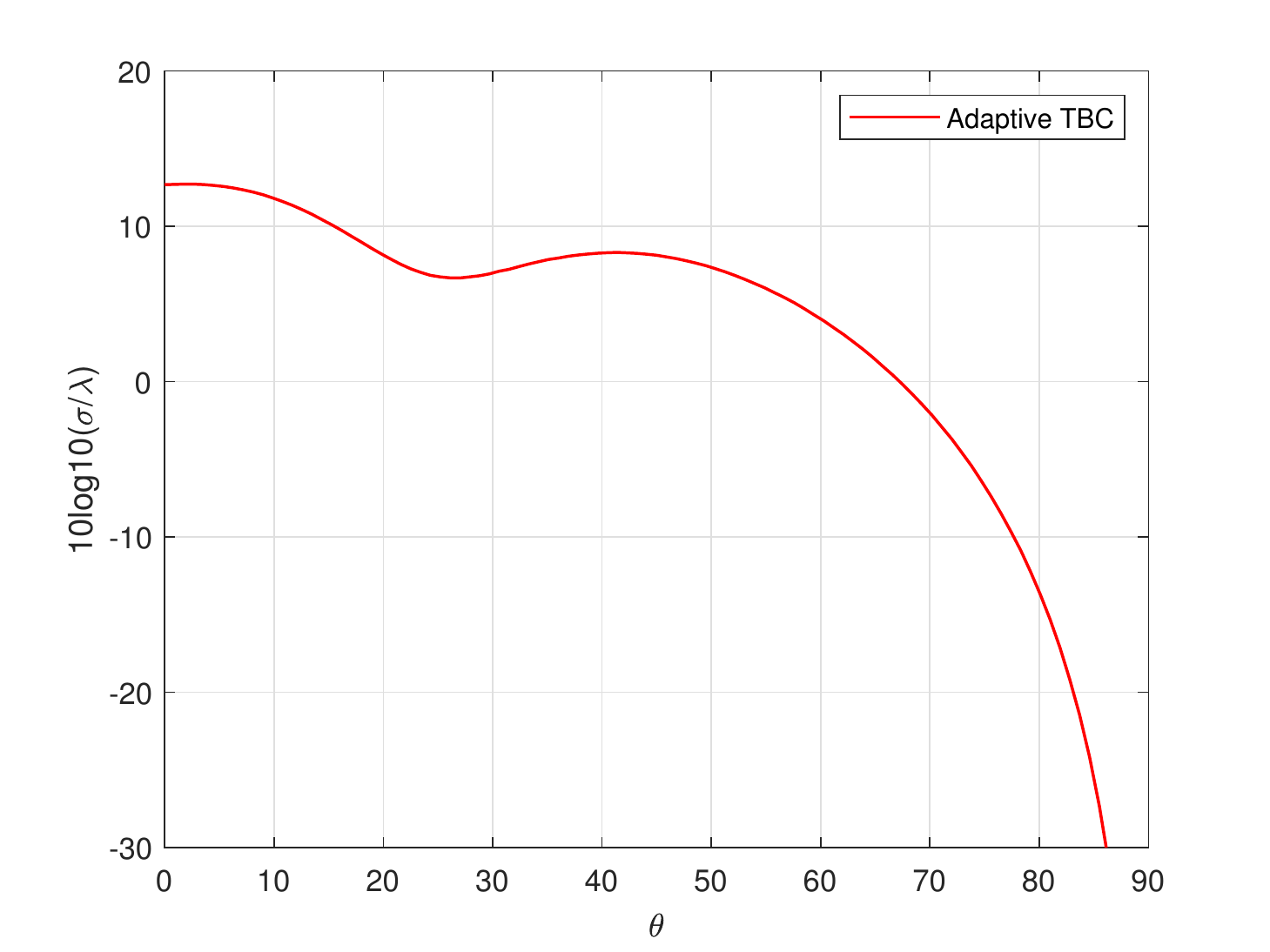}
\caption{Example 3: (left) Geometry of the cavity; (right) Backscatter RCS by
using the adaptive finite element DtN method.}
\label{TM_E3_G}
\end{figure}

\begin{figure}
\centering
\includegraphics[width=0.45\textwidth]{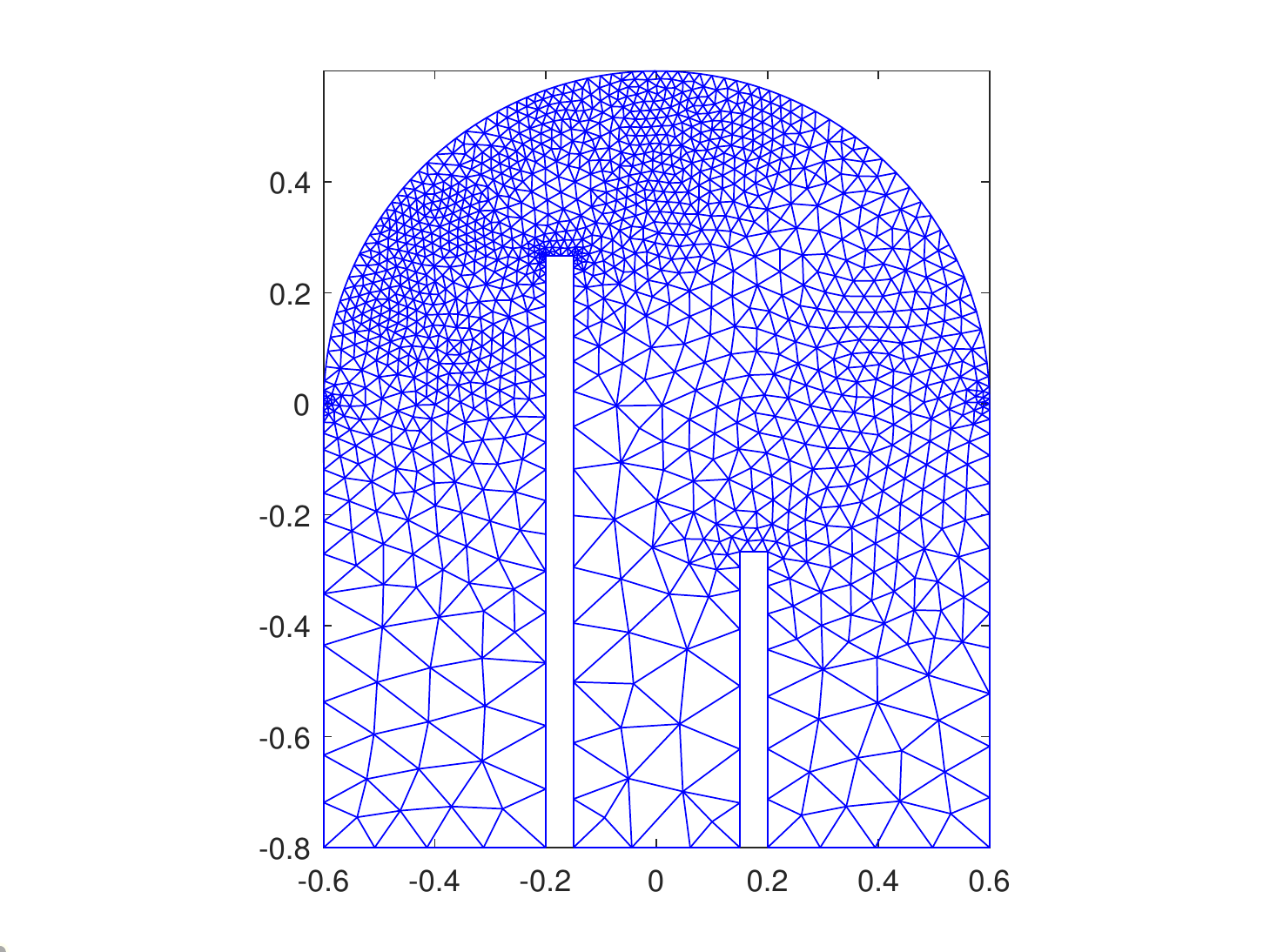}\quad
\includegraphics[width=0.45\textwidth]{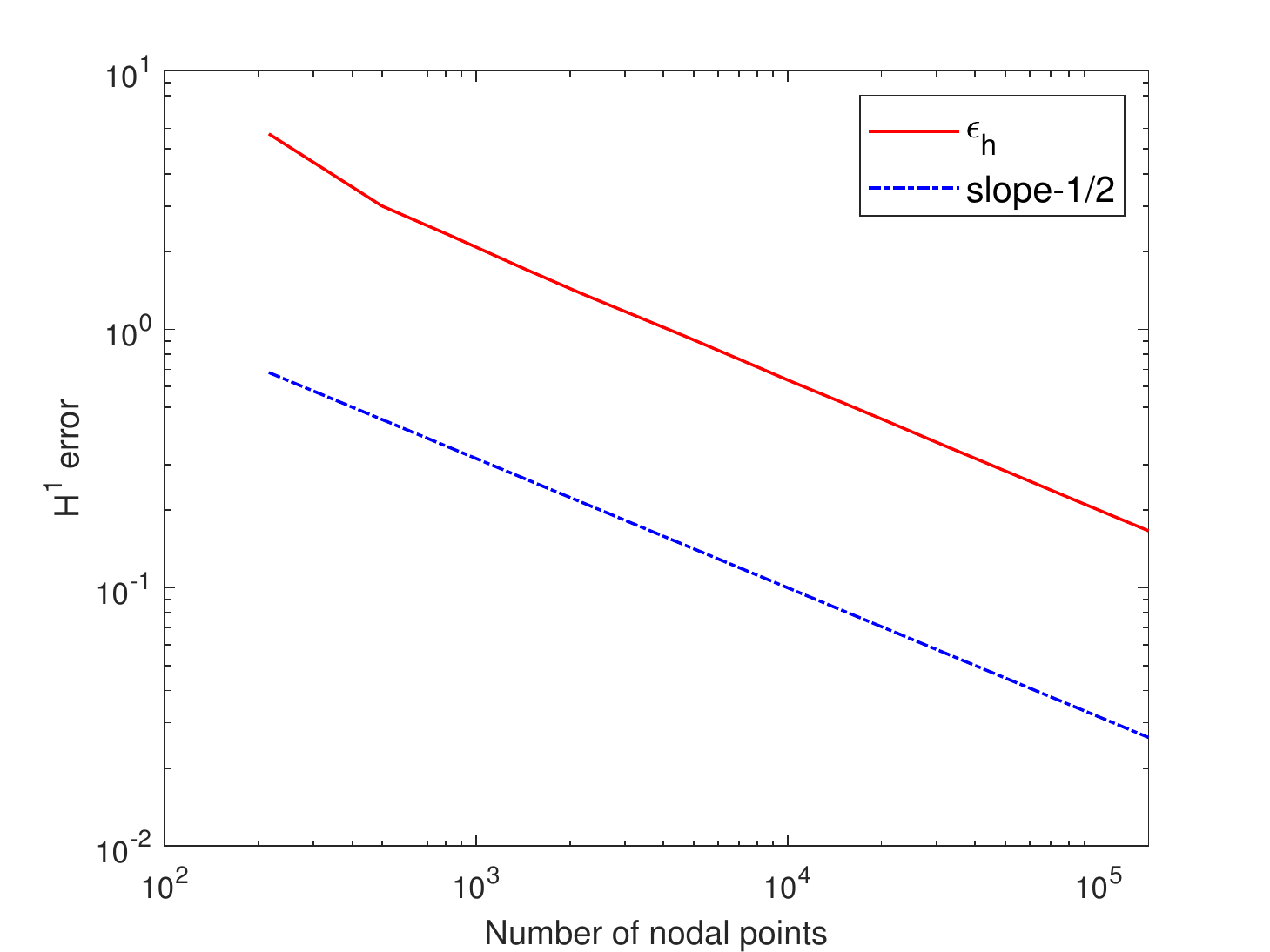}
\caption{Example 3: (left) Adaptive mesh after 2 iterations with a total number
of nodal points $1409$; (right) Quasi-optimality of the a posteriori error
estimates.}
\label{TM_E3_r}
\end{figure}

\subsection{Example 4}

In this example, we consider the cavity scattering problem for TE polarization.
The cavity is a rectangle with a fixed width $0.025 \rm m$ and a fixed depth
$0.015 \rm m$. The cavity is empty and filled with the same homogeneous medium
as that in the free space. Instead of considering the illumination by a plane
wave with a fixed frequency, we compute backscatter RCS with the frequency
ranging from $2\,\rm GHz$ to $18\, \rm GHz$. Correspondingly, the range
of the aperture of cavity is from $\frac{1}{6}\lambda$ to $1.5\lambda$.
The incident angle is $\frac{4}{9}\pi$. Figure \ref{TE_E1_RCS}
shows the backward RCS by using the adaptive finite element DtN method, where
the red-solid line and blue circles show the results obtained by applying
\eqref{RCS_TEf} and \eqref{RCS_TE}, respectively. The stopping criterion is
that the mesh refinement is stopped when the number of nodal points is over
25000.

\begin{figure}
\centering
\includegraphics[width=0.45\textwidth]{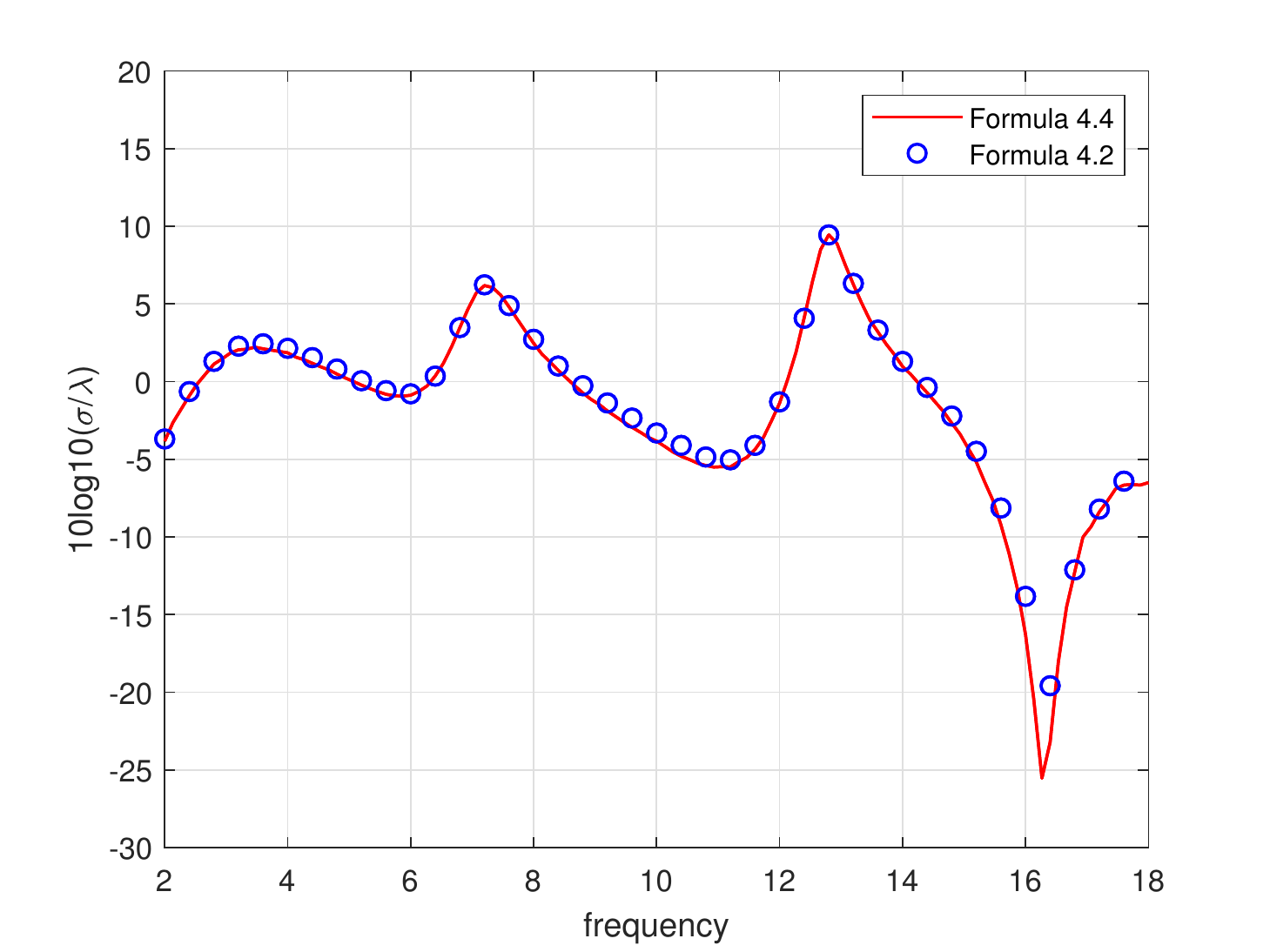}
\caption{Example 4: Backscatter RCS.}
\label{TE_E1_RCS}
\end{figure}

\section{Conclusion}\label{section:conclusion}

In this paper, we have developed an adaptive finite element DtN method for
solving the open cavity scattering problems. The a posteriori error estimates
are obtained for both of the TM and TE polarization waves. The estimates consist
of the finite element discretization error and the DtN operator truncation
error. The latter is shown to decay exponentially with respect to the
truncation number. Along the line of this research, future work includes
extending the analysis to the more challenging three-dimensional problem, where
Maxwell's equations need to be considered, and the problems of elastic wave
scattering from cavities. An open problem is to develop a DtN based TBC on the
upper semi-circle enclosing the cavities. We hope to report the progress on
these aspects elsewhere in the future.

\end{document}